\documentclass[11pt]{article}

\usepackage[T1]{fontenc}
\usepackage[utf8]{inputenc}
\usepackage[a4paper,textwidth=126mm,textheight=195mm,centering,includeheadfoot]{geometry}
\usepackage{amsmath,amssymb,amsfonts,amsthm,mathtools}
\usepackage{bm}
\usepackage{enumitem}
\usepackage{microtype}
\usepackage{natbib}
\usepackage[hidelinks]{hyperref}
\usepackage[nameinlink,noabbrev]{cleveref}
\usepackage{fancyhdr}
\ifdefined\reviewversion
  \usepackage[mathlines]{lineno}
\fi

\allowdisplaybreaks
\numberwithin{equation}{section}
\setcitestyle{authoryear,round,semicolon}

\newtheorem{theorem}{Theorem}[section]
\newtheorem{proposition}[theorem]{Proposition}
\newtheorem{lemma}[theorem]{Lemma}
\newtheorem{corollary}[theorem]{Corollary}
\newtheorem{definition}[theorem]{Definition}

\newtheorem{example}[theorem]{Example}
\newtheorem{counterexample}[theorem]{Counterexample}
\theoremstyle{remark}
\newtheorem{remark}[theorem]{Remark}

\newcommand{\R}{\mathbb{R}}
\newcommand{\C}{\mathbb{C}}

\newcommand{\Id}{I}
\newcommand{\Dg}{D_g}
\newcommand{\Cg}{C_g}
\newcommand{\mug}{\mu_g}
\newcommand{\mugc}{\mu_g^{\mathrm c}}
\newcommand{\ACg}{AC_g}

\newcommand{\Lg}[1]{L_g^{#1}}
\newcommand{\dmu}{\,d\mu_g}
\newcommand{\dmuc}{\,d\mu_g^{\mathrm c}}
\newcommand{\Deltag}{\Delta g}
\newcommand{\Range}{\operatorname{Ran}}
\newcommand{\Ker}{\operatorname{Ker}}
\newcommand{\rank}{\operatorname{rank}}
\newcommand{\Span}{\operatorname{span}}
\newcommand{\essspan}{\operatorname*{ess\,span}_{\mu_g}}

\newcommand{\cR}{\mathcal{R}}
\newcommand{\cK}{\mathcal{K}}

\newcommand{\cU}{\mathcal{U}}
\newcommand{\cC}{\mathcal{C}}

\newcommand{\norm}[1]{\left\lVert #1\right\rVert}
\newcommand{\abs}[1]{\left|#1\right|}
\newcommand{\ip}[2]{\left\langle #1,#2\right\rangle}
\newcommand{\eG}{\mathcal E_g}

\title{Reachability and exact controllability of linear Stieltjes differential systems}
\author{Francisco J. Fern\'andez\\[0.6ex]
\small\textit{Departamento de Estat\'istica, An\'alise Matem\'atica e Optimizaci\'on,}\\
\small\textit{Facultade de Matem\'aticas, Campus Vida,}\\
\small\textit{Universidade de Santiago de Compostela,}\\
\small\textit{15782 Santiago de Compostela, Spain}\\
\small\textit{CITMAga, 15782 Santiago de Compostela, Spain}\\[0.5ex]
\small Corresponding author. Email: \href{mailto:fjavier.fernandez@usc.es}{fjavier.fernandez@usc.es}\\
\small ORCID: \href{https://orcid.org/0000-0002-0667-2639}{0000-0002-0667-2639}}
\date{}

\hypersetup{
 pdftitle={Reachability and exact controllability of linear Stieltjes differential systems},
 pdfauthor={Francisco J. Fernandez},
 pdfkeywords={Stieltjes differential equations; reachability; exact controllability; atomic measures; Kalman condition; minimum-energy control}
}

\begin{document}
\maketitle
\thispagestyle{plain}

\begin{abstract}
We study reachability and exact controllability for systems driven by a
nondecreasing left-continuous Stieltjes derivator whose measure may have continuous,
singular-continuous and atomic parts, while flat intervals represent inactivity. We
construct an evolution family and post-jump variation-of-constants formula, and
characterize the reachable subspace as the essential span of input directions. For
constant coefficients, the Kalman condition is necessary for every derivator and,
under regressivity, sufficient when the derivator is continuous and nonconstant on a
subinterval. If the measure is purely atomic with exactly \(N\) positive atoms, the
reachable subspace is the truncated Krylov space generated by
\(B,AB,\ldots,A^{N-1}B\); for a Kalman-controllable pair, the minimum atom count among
purely atomic derivators equals the controllability index. At the threshold count, one
vanishing mass with the others converging to positive finite limits causes Gramian
degeneration and worst-case energy blow-up; at or above threshold, compact positive
mass ranges yield uniform coercivity. A mixed double-integrator example illustrates the results.
\end{abstract}



\ifdefined\reviewversion
\linenumbers
\fi

\section{Introduction}

Stieltjes differential equations provide a measure-based formulation for systems whose
evolution combines continuous phases, inactive intervals and instantaneous changes. A
nondecreasing left-continuous function \(g\), called a derivator, determines the
Lebesgue--Stieltjes measure \(\mu_g\). Its non-atomic part describes continuous
evolution, flat portions of \(g\) represent inactive phases and its atoms generate jump
relations. This setting contains ordinary differential equations, difference equations
and impulsive systems as particular cases \citep{LopezRodriguez2015,FrigonPouso2017}.

The analytic foundations of Stieltjes differential equations include the fundamental
theorem of Stieltjes calculus, pointwise and higher-order differentiation, the
nonclassical product rule and Stieltjes exponentials
\citep{FrigonPouso2017,FernandezMarquezTojo2022,CoraFernandezTojo2023,FernandezMarquezTojoOpen2024,FernandezMarquezTojoProduct2025}.
Compactness properties of the associated function spaces were studied by
\citet{FernandezTojoVillanueva2024}. We work throughout with the canonical space of
\(g\)-absolutely continuous functions. This fixes the representative of a solution and
avoids the nonuniqueness that may arise in larger spaces because the pointwise
Stieltjes derivative can have a nontrivial kernel
\citep{FernandezEtAlKernel2025}.

Controllability on nonstandard time domains has an extensive literature. Dynamic
systems on time scales admit transition matrices, shifted input kernels,
controllability Gramians and Kalman-type criteria
\citep{BohnerPeterson2001,DavisEtAl2009}. Optimal control on arbitrary time scales was
placed in a general nonlinear framework by \citet{BourdinTrelat2013}, who proved a
strong Pontryagin maximum principle and showed that right-dense and right-scattered
points lead to genuinely different first-order conditions: Hamiltonian maximization
holds almost everywhere at right-dense points, whereas at right-scattered points only
a directional condition is available without additional convexity. Closely related
constructions occur in impulsive systems and in generalized or measure-driven
differential equations \citep{LakshmikanthamBainovSimeonov1989,Schwabik1992}. A first
controllability theory for Stieltjes systems was proposed by
\citet{SiWangFeckan2020}. These precedents show that the existence of a transition
matrix, a Gramian, or a maximum principle is not, by itself, the relevant novelty. The
question addressed here is instead how the temporal structure encoded by \(\mu_g\)
determines which control directions are visible at the terminal time and how it
affects the associated minimum-energy problem.

For a fixed constant pair \((A,B)\), the classical Kalman space describes the
directions obtained by repeated action of \(A\) on the input space. The Stieltjes
derivator supplies a second and independent resource: the family of propagated input
directions that appears in the endpoint operator. An atom-free interval on which
the derivator varies continuously produces a continuum of such directions, a single atom provides only one,
and a finite set of atoms generates a finite-horizon Krylov structure. The paper makes
this distinction precise and quantifies its consequences.

\subsection{Main contributions and relation to the literature}

The first contribution is a measure-dependent characterization of reachability. For
\(A\in L_g^1\), we construct the canonical evolution family. Under \(g\)-regressivity
we derive the post-jump variation-of-constants formula for \(L_g^1\) forcing. With
\(B\in L_g^2\) and controls in \(L_g^2\), the endpoint operator then gives
\[
 \mathcal R_g(t_0,t_1)
 =\operatorname*{ess\,span}_{\mu_g}
 \left\{\operatorname{Ran}\bigl(\Phi(t_1,s^+)B(s)\bigr):s\in[t_0,t_1)\right\}.
\]
This identity treats absolutely continuous, singular-continuous and atomic temporal
components in a common Hilbert-space framework. In particular, the post-jump argument
\(s^+\) and the atomic weights in the Gramian follow from the Stieltjes product rule
and the canonical jump identity; they are not additional modelling conventions.

The second contribution is an exact classification for finite atomic derivators. For
constant matrices \(A\) and \(B\), the Kalman condition is necessary for every
derivator and, under the regressivity hypothesis of
Theorem~\ref{thm:continuous-window}, becomes sufficient whenever \(g\) has a continuous
nonconstant restriction to a nondegenerate subinterval. If \(\mu_g\) is purely atomic
and consists of exactly \(N\) positive atoms, we prove
\[
 \mathcal R_g(t_0,t_1)
 =\operatorname{span}\{\operatorname{Ran}B,\operatorname{Ran}AB,\ldots,
 \operatorname{Ran}A^{N-1}B\}.
\]
Consequently, for a Kalman-controllable pair, the least number of atoms among purely
atomic derivators that yields exact controllability is its controllability index. This
conclusion holds for every choice of positive masses and is independent of the
locations of the atoms.

The third contribution separates algebraic reachability from control energy. Although
positive atomic masses do not alter the reachable subspace in the preceding
classification, they enter the Gramian as weights. For a fixed Kalman-controllable
pair with \(N=\nu(A,B)\), we show that the smallest Gramian eigenvalue tends to zero
when exactly one mass tends to zero and the others converge to positive finite limits;
therefore the worst-case minimum energy diverges. For a fixed Kalman-controllable pair
and \(N\ge\nu(A,B)\), masses in a compact subset of \((0,\infty)^N\) yield a uniform
positive Gramian lower bound. Thus the number of active temporal events determines
reachability, whereas their masses determine the conditioning of the control problem.

The evolution-family and Gramian results developed below provide the analytic
infrastructure for these three conclusions. They are stated in detail because the
continuous, singular-continuous and atomic parts of \(\mu_g\) must be handled
simultaneously and because the forcing term at an atom is propagated from the
post-jump state.

The paper is organized as follows. Section~\ref{sec:prelim} introduces the Stieltjes
framework and the canonical solution space. Section~\ref{sec:comparison} compares the
present formulation with continuous-time, time-scale, impulsive and measure-driven
systems. Section~\ref{sec:linear} develops the evolution-family and
variation-of-constants theory. Section~\ref{sec:control} contains the endpoint,
Gramian and reachability results, including the continuous-component criterion, the
finite-atomic classification and the energy estimates. The final section summarizes
the conclusions and directions for further work.

\section{Preliminaries}\label{sec:prelim}

Let \(I=[t_0,T]\), where \(t_0<T\), and let
\(g:\R\to\R\) be nondecreasing and left-continuous. We denote
\[
 \Deltag(t):=g(t^+)-g(t),\qquad
 \Dg:=\{t\in\R:\Deltag(t)>0\},
\]
and
\[
 \Cg:=\{t\in\R:\text{\(g\) is constant on a neighborhood of \(t\)}\}.
\]
The Lebesgue--Stieltjes measure associated with \(g\) is denoted by \(\mug\) and is
normalized by
\[
 \mug([a,b))=g(b)-g(a),\qquad a<b.
\]
In particular,
\[
 \mug(\{t\})=\Deltag(t),\qquad t\in\Dg.
\]
The restriction of \(\mug\) to \([t_0,T)\) is finite. On this interval its
atomic--non-atomic decomposition is
\[
 \mug|_{[t_0,T)}=\mugc|_{[t_0,T)}
 +\sum_{\tau\in\Dg\cap[t_0,T)}\Deltag(\tau)\delta_\tau,
\]
where \(\mugc\) denotes the non-atomic part; see, e.g., \citep{Bogachev2007}.
For the present normalization the atomic mass at \(\tau<T\) is exactly
\(\Deltag(\tau)\).
The set \(\Dg\cap[t_0,T)\) is at most countable and
\[
 \sum_{\tau\in\Dg\cap[t_0,T)}\Deltag(\tau)\le g(T)-g(t_0).
\]

The measure \(\mug\) is understood on its completion. Whenever an
\(L_g^p\)-coefficient is used pointwise, we choose a Borel representative. This causes
no ambiguity at an atom, because a singleton \(\{\tau\}\), \(\tau\in\Dg\), has positive
\(\mug\)-measure and hence the value of an \(L_g^p\)-class at \(\tau\) is fixed.

The pseudometric induced by the derivator is
\[
 d_g(s,t):=\abs{g(t)-g(s)}.
\]
A function \(x:I\to\R^n\) is called \(g\)-continuous if it is continuous from
\((I,d_g)\) to \(\R^n\). Equivalently, for every \(t\in I\) and every
\(\varepsilon>0\), there is \(\delta>0\) such that
\[
 \abs{g(s)-g(t)}<\delta
 \quad\Longrightarrow\quad
 \norm{x(s)-x(t)}<\varepsilon.
\]
This is the standard continuity notion associated with a derivator; see
\cite[Definition~2.1 and Proposition~2.3]{FernandezTojoVillanueva2024}.

All vector spaces are endowed with their Euclidean inner products. Matrix norms are
induced operator norms unless another norm is stated explicitly. In particular, for a
symmetric positive matrix \(W\), \(\lambda_{\min}(W)\) denotes its smallest eigenvalue
and \(\norm{W^{-1}}=1/\lambda_{\min}(W)\).

\begin{example}[Inactive intervals]\label{ex:plateau}
Let
\[
 g(t)=
 \begin{cases}
 t,&t\le1,\\
 1,&1<t\le2,\\
 t-1,&t>2.
 \end{cases}
\]
Then \(g\) is continuous and nondecreasing, while \((1,2)\subset\Cg\). The canonical
solution of
\[
 x'_g(t)=1,\qquad x(0)=0,
\]
is \(x(t)=g(t)-g(0)\), and is therefore constant on \([1,2]\). Thus the set \(\Cg\) describes an actual inactive phase of the system. In the measure formulation this property follows directly from \(\mug((1,2))=0\).
\end{example}

\subsection{Canonical absolute continuity}

We adopt the following integral definition, which determines a canonical representative and avoids separate conventions at the endpoints of the components of \(\Cg\).

\begin{definition}\label{def:acg}
A function \(x:I\to\R^n\) belongs to \(\ACg(I;\R^n)\) if there exists
\(v\in\Lg{1}([t_0,T);\R^n)\) such that
\begin{equation}\label{eq:canonical-representation}
 x(t)=x(t_0)+\int_{[t_0,t)}v(s)\dmu(s),
 \qquad t\in I.
\end{equation}
The function \(v\), uniquely determined \(\mug\)-almost everywhere, is denoted by
\(x'_g\).
\end{definition}

The equivalence between this integral representation and the usual definition of
Stieltjes absolute continuity is the fundamental theorem of Stieltjes calculus; see
\cite[Proposition~5.4]{FrigonPouso2017} and the convenient restatement
\cite[Theorem~2.12]{FernandezMarquezTojoOpen2024}. Since every atom has positive
\(\mug\)-measure, the value of \(x'_g\) at an atom is determined by its
\(L_g^1\)-equivalence class.

\begin{proposition}[Canonical representative and jump law]\label{prop:canonical}
Let \(x\in\ACg(I;\R^n)\). Then \(x\) is bounded and \(g\)-continuous, and it is
constant on every interval on which \(g\) is constant. Moreover, for every
\(t\in[t_0,T)\cap\Dg\),
\begin{equation}\label{eq:canonical-jump}
 x(t^+)-x(t)=\Deltag(t)x'_g(t).
\end{equation}
\end{proposition}

\begin{proof}
Let \(v=x'_g\) be the integrand in \eqref{eq:canonical-representation}. For every
\(t\in I\),
\[
 \norm{x(t)}
 \le \norm{x(t_0)}+\int_{[t_0,T)}\norm{v(r)}\dmu(r),
\]
so \(x\) is bounded. If \(s<t\), subtraction of the two integral representations
gives the exact increment formula
\begin{equation}\label{eq:AC-increment}
 x(t)-x(s)=\int_{[s,t)}v(r)\dmu(r),
\end{equation}
and hence
\begin{equation}\label{eq:AC-modulus}
 \norm{x(t)-x(s)}
 \le \int_{[s,t)}\norm{v(r)}\dmu(r).
\end{equation}

Fix \(\varepsilon>0\). By absolute continuity of the Lebesgue integral, there is
\(\delta>0\) such that
\[
 \mug(E)<\delta
 \quad\Longrightarrow\quad
 \int_E\norm{v(r)}\dmu(r)<\varepsilon
\]
for every measurable \(E\subset[t_0,T)\). If \(s<t\) and
\(d_g(s,t)=g(t)-g(s)<\delta\), then
\(\mug([s,t))<\delta\), and \eqref{eq:AC-modulus} yields
\(\norm{x(t)-x(s)}<\varepsilon\). The case \(t<s\) is identical. Thus \(x\) is
\(g\)-continuous.

If \(g\) is constant on \([s,t]\), then \(\mug([s,t))=0\), so
\eqref{eq:AC-increment} gives \(x(t)=x(s)\). Finally, let
\(\tau\in[t_0,T)\cap\Dg\). For \(r>\tau\),
\[
 x(r)-x(\tau)=\int_{[\tau,r)}v(s)\dmu(s)
 =\Deltag(\tau)v(\tau)+\int_{(\tau,r)}v(s)\dmu(s).
\]
As \(r\downarrow\tau\), the sets \((\tau,r)\) decrease to the empty set. Since
\(v\in L_g^1\), the second integral tends to zero by continuity from above of the
finite measure \(\norm{v}\mug\). Therefore the right limit exists and
\[
 x(\tau^+)-x(\tau)=\Deltag(\tau)v(\tau),
\]
which is \eqref{eq:canonical-jump}.
\end{proof}

The canonical space \(\ACg\) is also the appropriate setting for uniqueness. In the larger space of functions that
are merely Stieltjes differentiable at every point, the kernel of the derivative may
contain nonconstant functions; see \citep{FernandezEtAlKernel2025}.

For \(x\in\ACg\) and \(t\in[t_0,T)\), define
\[
 x^+(t):=
 \begin{cases}
 x(t^+),&t\in\Dg,\\
 x(t),&t\notin\Dg.
 \end{cases}
\]
By \eqref{eq:canonical-jump},
\[
 x^+(t)=x(t)+\Deltag(t)x'_g(t)
 \quad\text{for \(\mug\)-almost every \(t\)}.
\]

\begin{proposition}[Product rule]\label{prop:product}
Let \(U\) and \(V\) be compatible vector- or matrix-valued functions in \(\ACg(I)\).
Then \(UV\in\ACg(I)\) and
\begin{align}
 (UV)'_g(t)
 &=
 U'_g(t)V(t)+U^+(t)V'_g(t) \label{eq:product-right}\\
 &=
 U'_g(t)V(t)+U(t)V'_g(t)
 +\Deltag(t)U'_g(t)V'_g(t)                 \label{eq:product-atomic}
\end{align}
for \(\mug\)-almost every \(t\).
\end{proposition}

\begin{proof}
For scalar-valued functions, the Stieltjes integration-by-parts identity states that
\begin{equation}\label{eq:IBP-Stieltjes}
 U(t)V(t)-U(s)V(s)
 =\int_{[s,t)}U'_g(r)V(r)\dmu(r)
  +\int_{[s,t)}U^+(r)V'_g(r)\dmu(r).
\end{equation}
It follows from the fundamental theorem and the pointwise product rule; see
\cite[Proposition~3.9]{FernandezMarquezTojo2022} and
\cite[Lemma~2.13]{FernandezMarquezTojoOpen2024}. Applying
\eqref{eq:IBP-Stieltjes} componentwise gives the same identity for compatible
vector- and matrix-valued functions. The right-hand side belongs to \(L_g^1\): indeed,
\(U,V\in\ACg\) are bounded by Proposition~\ref{prop:canonical}, whereas
\(U'_g,V'_g\in L_g^1\). Consequently \(UV\in\ACg\), and its
Radon--Nikodym derivative is
\[
 (UV)'_g=U'_gV+U^+V'_g
 \quad\mug\text{-a.e.}
\]
This proves \eqref{eq:product-right}. The jump identity
\(U^+=U+\Deltag\,U'_g\), established in Proposition~\ref{prop:canonical}, then gives
\[
 U'_gV+U^+V'_g
 =U'_gV+UV'_g+\Deltag\,U'_gV'_g,
\]
which is \eqref{eq:product-atomic}.
\end{proof}

The extra term in \eqref{eq:product-atomic} is concentrated on \(\Dg\). On the
non-atomic part of \(\mug\), \(\Deltag=0\) and the formula reduces to the classical
product rule. The higher-order regularity consequences of this correction are studied
in detail in \citep{FernandezMarquezTojoProduct2025}.

\begin{remark}[The natural regularity is measured by \(g\)]\label{rem:not-usual-lipschitz}
If \(t\in\Dg\) and \(x'_g(t)\ne0\), then
\[
 x(t^+)-x(t)=\Deltag(t)x'_g(t)\ne0.
\]
Hence a nontrivial canonical solution may have a jump and cannot be locally Lipschitz
with respect to the ordinary distance \(\abs{s-t}\). Statements of regularity for the
evolution are therefore formulated in \(\ACg\), or equivalently through estimates in
the pseudodistance \(\abs{g(s)-g(t)}\), rather than through ordinary local
Lipschitzness.
\end{remark}

\begin{lemma}[Substitution for a continuous Stieltjes primitive]
\label{lem:continuous-cumulative-substitution}
Let \(Q:[a,b]\to\mathbb R\) be continuous and nondecreasing, and let \(\mu_Q\) be
the associated Lebesgue--Stieltjes measure, normalized by
\(\mu_Q([r_1,r_2))=Q(r_2)-Q(r_1)\). For every \(t\in[a,b]\) and every
nonnegative Borel function \(\psi\),
\begin{equation}\label{eq:continuous-cumulative-substitution}
 \int_{[a,t)}\psi(Q(r))\,d\mu_Q(r)
 =\int_{[Q(a),Q(t))}\psi(y)\,dy.
\end{equation}
The same identity holds for every integrable Borel function.
\end{lemma}

\begin{proof}
Let \(\nu_t\) be the pushforward of \(\mu_Q|_{[a,t)}\) by \(Q\). For
\(y\in(Q(a),Q(t))\), set
\[
 \alpha_y:=\inf\{r\in[a,t]:Q(r)\ge y\}.
\]
Continuity and monotonicity imply \(Q(\alpha_y)=y\) and
\(\{r\in[a,t):Q(r)<y\}=[a,\alpha_y)\). Therefore
\[
 \nu_t([Q(a),y))
 =\mu_Q([a,\alpha_y))
 =Q(\alpha_y)-Q(a)=y-Q(a).
\]
The same relation is immediate at the two endpoints. Hence \(\nu_t\) and Lebesgue
measure agree on the half-open intervals generating the Borel sets of
\([Q(a),Q(t))\). Uniqueness of finite measures gives
\(\nu_t=\mathcal L^1|_{[Q(a),Q(t))}\), and
\eqref{eq:continuous-cumulative-substitution} follows from the definition of a
pushforward integral. The integrable case follows by applying the nonnegative result
to the positive and negative parts. This is the standard substitution theorem for a
continuous Lebesgue--Stieltjes distribution function; see, for example,
\citep{Bogachev2007}.
\end{proof}

\subsection{The positive \texorpdfstring{\(g\)}{g}-exponential}

Let \(q\in\Lg{1}([t_0,T);[0,\infty))\). For \(t_0\le s\le t\le T\), define
\begin{equation}\label{eq:positive-gexp}
 E_q(t,s)
 :=
 \exp\!\left(\int_{[s,t)}q(r)\dmuc(r)\right)
 \prod_{\tau\in[s,t)\cap\Dg}
 \bigl(1+q(\tau)\Deltag(\tau)\bigr).
\end{equation}
The product converges because
\[
 \sum_{\tau\in[s,t)\cap\Dg}
 q(\tau)\Deltag(\tau)
 \le \int_{[s,t)}q(r)\dmu(r)<\infty.
\]
Moreover,
\begin{equation}\label{eq:gexp-bound}
 1\le E_q(t,s)
 \le
 \exp\!\left(\int_{[s,t)}q(r)\dmu(r)\right).
\end{equation}

\begin{proposition}[Scalar comparison equation]\label{prop:gexp}
The function \(E_q(\cdot,s)\) is the unique element of \(\ACg([s,T])\) satisfying
\begin{equation}\label{eq:gexp-integral}
 E_q(t,s)
 =
 1+\int_{[s,t)}q(r)E_q(r,s)\dmu(r).
\end{equation}
At an atom \(\tau\),
\[
 E_q(\tau^+,s)
 =
 \bigl(1+q(\tau)\Deltag(\tau)\bigr)E_q(\tau,s).
\]
\end{proposition}

\begin{proof}
Enumerate the atoms of \([s,T)\) as \(\{\tau_k\}_{k\ge1}\), and define
\[
 \mug^{(N)}:=\mugc+\sum_{k=1}^N\Deltag(\tau_k)\delta_{\tau_k},
\]
\[
 E_q^{(N)}(t,s)
 :=\exp\!\left(\int_{[s,t)}q(r)\dmuc(r)\right)
   \prod_{\substack{k\le N\\ \tau_k<t}}
   \bigl(1+q(\tau_k)\Deltag(\tau_k)\bigr).
\]
To justify the continuous factor, set
\[
 Q(t):=\int_{[s,t)}q(r)\dmuc(r).
\]
Then \(Q\) is continuous and nondecreasing, and its Lebesgue--Stieltjes measure is
\(d\mu_Q=q\,d\mugc\). Applying
Lemma~\ref{lem:continuous-cumulative-substitution} with \(\psi(y)=e^y\) gives
\[
 \int_{[s,t)}e^{Q(r)}\,d\mu_Q(r)
 =\int_{[0,Q(t))}e^y\,dy=e^{Q(t)}-1.
\]
Since \(d\mu_Q=q\,d\mugc\), this is precisely
\[
 \exp\!\left(\int_{[s,t)}q\dmuc\right)
 =1+\int_{[s,t)}q(r)
   \exp\!\left(\int_{[s,r)}q\dmuc\right)\dmuc(r).
\]
This substitution formula is a standard Lebesgue--Stieltjes identity; see
\citep{Bogachev2007}. Crossing an included atom \(\tau_k\) multiplies
\(E_q^{(N)}\) by
\(1+q(\tau_k)\Deltag(\tau_k)\). Hence, by the fundamental theorem of Stieltjes
calculus,
\begin{equation}\label{eq:gexp-truncated}
 E_q^{(N)}(t,s)
 =1+\int_{[s,t)}q(r)E_q^{(N)}(r,s)\,d\mug^{(N)}(r).
\end{equation}

Since \(\sum_k q(\tau_k)\Deltag(\tau_k)<\infty\), the infinite product in
\eqref{eq:positive-gexp} converges uniformly in \(t\), and
\(E_q^{(N)}(\cdot,s)\to E_q(\cdot,s)\) uniformly. Moreover,
\[
 0\le E_q^{(N)}(t,s)\le
 \exp\!\left(\int_{[s,T)}q(r)\dmu(r)\right)=:M_q.
\]
Using \eqref{eq:gexp-truncated}, we estimate the difference between the truncated and
full integrals by
\begin{align*}
 &\left|\int qE_q^{(N)}\,d\mug^{(N)}
       -\int qE_q\,d\mug\right|\\
 &\qquad\le
 \int_{[s,T)}q(r)\abs{E_q^{(N)}(r,s)-E_q(r,s)}\dmu(r)
 +M_q\sum_{k>N}q(\tau_k)\Deltag(\tau_k),
\end{align*}
which tends to zero. Passing to the limit proves \eqref{eq:gexp-integral}.
The atomic identity follows either from the product formula or directly from
Proposition~\ref{prop:canonical}.

For uniqueness, let \(Y_1,Y_2\in\ACg([s,T])\) solve
\eqref{eq:gexp-integral} and set \(y=Y_1-Y_2\). Then
\[
 y'_g=qy\quad\mug\text{-a.e.},\qquad y(s)=0.
\]
The function \(E_q(\cdot,s)\) is strictly positive. Put
\(R(t)=E_q(t,s)^{-1}\). Since \(E_q(\cdot,s)\in\ACg([s,T])\) takes values in the
compact interval \([1,M_q]\) and \(r\mapsto r^{-1}\) is Lipschitz there, the
composition theorem for measure-absolutely-continuous functions gives
\(R\in\ACg([s,T])\). Off the atoms, the ordinary chain rule with respect to the
continuous Stieltjes variable gives \(R'_g=-qR\). At an atom \(\tau\),
\[
 R'_g(\tau)
 =\frac{E_q(\tau^+,s)^{-1}-E_q(\tau,s)^{-1}}{\Deltag(\tau)}
 =-q(\tau)R(\tau^+).
\]
Thus \(R'_g=-qR^+\) almost everywhere. By the product rule,
\[
 (Ry)'_g=R'_g y+R^+y'_g=-qR^+y+R^+qy=0.
\]
Since \(Ry\in\ACg\), its integral representation shows that it is constant; its
value at \(s\) is zero. Hence \(y=0\), which proves uniqueness without invoking the
Peano--Baker construction below.
\end{proof}

This product-integral representation is the positive-coefficient case of the
Stieltjes exponential considered in \citep{FernandezMarquezTojo2022} and
\cite[Section~5]{CoraFernandezTojo2023}. We use it as a scalar comparison function in
the construction of the evolution family.

\section{Connections with related control frameworks}\label{sec:comparison}

The objects used below have analogues in continuous-time, discrete-time, time-scale,
impulsive, and measure-driven control. We record the precise correspondence both to
locate the present results and to identify the role of the Stieltjes measure.

\paragraph{Continuous-time systems.}
If \(g(t)=t\), then \(\mug\) is Lebesgue measure, \(\Dg=\varnothing\), and
\(\Phi(t,s^+)=\Phi(t,s)\). The endpoint operator, controllability Gramian, and
minimum-energy control reduce to the classical constructions; see, for example,
\citep{Kailath1980,Sontag1998}.

\paragraph{Discrete-time systems.}
If \(\mug\) is supported by finitely many atoms, the state evolves through the
recurrence
\[
 x_{j}^{+}=(I+\delta_jA)x_j+\delta_jBu_j.
\]
The finite-horizon reachable space is therefore generated by products of the step
matrices. Section~\ref{sec:control} shows that, for constant \(A,B\) and arbitrary
positive masses, these products form a Newton-type basis of the truncated Krylov
space. This is consistent with the classical finite-horizon reachability matrix, but
it also identifies the exact role of the Stieltjes masses and their effect on the
control energy.

\paragraph{Systems on time scales.}
Let \(\mathbb T\) be a time scale, with forward jump operator \(\sigma\) and
graininess \(\mu_{\mathbb T}(r)=\sigma(r)-r\). On the convex hull of \(\mathbb T\), the
forward projection
\[
 g_{\mathbb T}(t):=\inf\{r\in\mathbb T:r\ge t\}
\]
(with the usual endpoint convention) is a nondecreasing left-continuous derivator.
Under this embedding, every right-scattered point \(r\) becomes an atom with
\[
 \Delta g_{\mathbb T}(r)=\mu_{\mathbb T}(r),
\]
whereas right-dense points do not produce atoms. Correspondingly,
\[
 \sigma(r)\longleftrightarrow r^+,
 \qquad
 \mu_{\mathbb T}(r)\longleftrightarrow\Delta g(r),
 \qquad
 \Phi(t,\sigma(s))\longleftrightarrow\Phi(t,s^+).
\]
The time-scale controllability Gramian therefore uses the same shifted input direction
as the Stieltjes Gramian; see \cite[Theorem~2.2]{DavisEtAl2009} and
\citep{BohnerPeterson2001}. At the measure level, for measurable
\(E\subset\mathbb T\setminus\{\sup\mathbb T\}\), the delta measure satisfies
\[
 \mu_\Delta(E)
 =\mathcal L^1(E)
 +\sum_{r\in E\cap RS}\mu_{\mathbb T}(r),
\]
so it consists of a Lebesgue part and graininess-weighted atoms
\citep{BourdinTrelat2013}. The Stieltjes class is broader in that it also admits
arbitrary singular-continuous components and atomic masses that are not constrained to
equal geometric gap lengths; see \citep{LopezRodriguez2015} for the Stieltjes
unification of continuous, discrete, and impulsive calculus.

The same shift is visible in optimal control. In the Pontryagin principle of
\citet{BourdinTrelat2013}, the adjoint enters the Hamiltonian through
\(p^\sigma(s)=p(\sigma(s))\), as dictated by the time-scale Leibniz rule. At
right-dense points one obtains a Hamiltonian maximization condition, while at
right-scattered points the general nonconvex problem yields only a directional
first-order condition. This dense/scattered distinction is the optimal-control
counterpart of the non-atomic/atomic separation used here.

\paragraph{Impulsive systems.}
Classical impulsive equations prescribe a continuous flow together with separate jump
maps; see \citep{LakshmikanthamBainovSimeonov1989}. In the present framework both
mechanisms are generated by the single measure \(\mug\). At an atom \(\tau\),
\[
 x(\tau^+)=(I+\Deltag(\tau)A(\tau))x(\tau)
 +\Deltag(\tau)B(\tau)u(\tau).
\]
Thus the atomic mass occurs simultaneously in the transition, the endpoint operator,
and the \(L_g^2\)-cost.

\paragraph{Measure-driven systems.}
General measure differential equations are commonly formulated for bounded-variation
solutions and vector-valued driving measures; see \citep{Schwabik1992}. Here the
driving measure is the positive scalar Lebesgue--Stieltjes measure associated with a
derivator, and solutions are taken in the canonical space \(\ACg\). This specialization
provides the exact decomposition
\[
 \mug=\mugc+\sum_{\tau\in\Dg}\Deltag(\tau)\delta_\tau
\]
and allows the endpoint map to be analysed separately on the non-atomic and atomic
parts.

A complementary line of research studies controllability for semilinear
measure-driven evolution equations in Banach spaces. Sufficient conditions for
complete controllability under nonlocal conditions were obtained by
\citet{CaoSun2017} through measures of noncompactness and the M\"onch fixed point
theorem. Related non-autonomous integrodifferential systems were considered by
\citet{NiangEtAl2025}, while impulsive systems with multivalued nonlocal inclusions
and functional control systems with nonlocal conditions were subsequently treated by
\citet{VermaDabas2026} and \citet{VermaKumarDabas2026}, respectively. These works
address broad semilinear, mostly infinite-dimensional settings and provide sufficient
conditions in terms of mild solutions, evolution or resolvent families, compactness
methods and fixed point arguments. The present finite-dimensional linear theory is
complementary: it determines the exact range of the terminal operator, gives sharp
algebraic classifications for continuous and finite-atomic derivators, identifies the
minimum number of atoms among purely atomic derivators, and quantifies the loss of
Gramian coercivity and the associated worst-case energy blow-up. No extension of
those conclusions to the aforementioned semilinear settings is asserted here.

The preceding correspondences explain why the Gramian itself is not the only issue.
The main question in Section~\ref{sec:control} is which propagated input directions are
actually visible to \(\mug\). An atom-free interval with nonzero continuous \(g\)-variation supplies a
continuum of directions, whereas \(N\) atoms supply exactly a truncated Krylov space
of order \(N\).

\section{Linear evolution families}\label{sec:linear}

Consider the linear system
\begin{equation}\label{eq:linear-system}
 x'_g(t)=A(t)x(t)+f(t),
 \qquad \mug\text{-a.e. }t\in[t_0,T),
\end{equation}
where
\[
 A\in\Lg{1}([t_0,T);\R^{n\times n}),
 \qquad
 f\in\Lg{1}([t_0,T);\R^n).
\]

\subsection{Peano--Baker--Volterra construction}

We construct the evolution family by means of a Lebesgue--Stieltjes version of the Peano--Baker series; see \citep{BaakeSchlaegel2011} for the classical construction.
For \(t_0\le s\le t\le T\), define recursively
\begin{align}
 \Phi_0(t,s)&:=\Id, \label{eq:PB0}\\
 \Phi_{k+1}(t,s)&:=
 \int_{[s,t)}A(r)\Phi_k(r,s)\dmu(r),
 \qquad k\ge0.                                     \label{eq:PBrec}
\end{align}

\begin{theorem}[Existence of the evolution family]\label{thm:evolution}
The series
\begin{equation}\label{eq:PBseries}
 \Phi(t,s):=\sum_{k=0}^\infty\Phi_k(t,s)
\end{equation}
converges uniformly in \(t\in[s,T]\). It is the unique solution of
\begin{equation}\label{eq:fundamental-volterra}
 \Phi(t,s)
 =
 \Id+\int_{[s,t)}A(r)\Phi(r,s)\dmu(r).
\end{equation}
The following properties hold.

\begin{enumerate}[label=\textnormal{(\roman*)}]
\item \(\Phi(\cdot,s)\in\ACg([s,T];\R^{n\times n})\) and
\[
 \partial_g\Phi(t,s)=A(t)\Phi(t,s)
 \quad\mug\text{-a.e.};
\]

\item
\begin{equation}\label{eq:evolution-bound}
 \norm{\Phi(t,s)}
 \le E_{\norm{A}}(t,s)
 \le
 \exp\!\left(\int_{[s,t)}\norm{A(r)}\dmu(r)\right);
\end{equation}

\item for \(s\le r\le t\),
\begin{equation}\label{eq:cocycle}
 \Phi(t,r)\Phi(r,s)=\Phi(t,s);
\end{equation}

\item at every \(\tau\in[s,T)\cap\Dg\),
\begin{equation}\label{eq:matrix-jump}
 \Phi(\tau^+,s)
 =
 \bigl(\Id+\Deltag(\tau)A(\tau)\bigr)\Phi(\tau,s).
\end{equation}
\end{enumerate}
Consequently, \(\Phi\) is uniformly bounded on
\(\{(t,s):t_0\le s\le t\le T\}\).
\end{theorem}

\begin{proof}
Fix \(s\in[t_0,T]\), and put \(a(t)=\norm{A(t)}\). Define the scalar Volterra
iterates
\[
 p_0(t,s)=1,\qquad
 p_{k+1}(t,s)=\int_{[s,t)}a(r)p_k(r,s)\dmu(r).
\]
We first prove by induction that
\begin{equation}\label{eq:PB-majorization}
 \norm{\Phi_k(t,s)}\le p_k(t,s),
 \qquad k\ge0.
\end{equation}
The assertion is clear for \(k=0\). If it holds for \(k\), then
\[
 \norm{\Phi_{k+1}(t,s)}
 \le\int_{[s,t)}a(r)\norm{\Phi_k(r,s)}\dmu(r)
 \le p_{k+1}(t,s).
\]

The partial sums \(P_N=\sum_{k=0}^Np_k\) increase pointwise. We first show
that they are bounded by \(E_a\). From
\[
 E_a(t,s)=1+\int_{[s,t)}a(r)E_a(r,s)\dmu(r)
\]
and positivity, one obtains successively
\[
 E_a(t,s)\ge p_0(t,s)+p_1(t,s),
\]
and, by substituting the same integral identity in the nonnegative remainder,
\[
 E_a(t,s)\ge\sum_{k=0}^{N}p_k(t,s)=P_N(t,s)
 \qquad (N\ge0).
\]
Therefore \(P_N\uparrow P\le E_a<\infty\). Monotone convergence now gives
\[
 P(t,s)=1+\int_{[s,t)}a(r)P(r,s)\dmu(r).
\]
The uniqueness part of Proposition~\ref{prop:gexp} yields \(P=E_a\). In particular,
\[
 \sum_{k=0}^{\infty}p_k(T,s)=E_a(T,s)<\infty.
\]
Since \(p_k(t,s)\le p_k(T,s)\), the Weierstrass test and
\eqref{eq:PB-majorization} imply uniform convergence of
\(\sum_k\Phi_k(\cdot,s)\) on \([s,T]\). Denote the limit by \(\Phi(\cdot,s)\).
The uniform estimate
\[
 \norm{\Phi(t,s)}
 \le\sum_{k=0}^{\infty}p_k(t,s)=E_a(t,s)
\]
proves the first inequality in \eqref{eq:evolution-bound}; the second is
\eqref{eq:gexp-bound}.

Each \(\Phi_k(\cdot,s)\) is measurable. Furthermore,
\[
 \norm{A(r)\Phi_k(r,s)}\le a(r)p_k(r,s),
\]
and
\[
 \sum_{k=0}^{\infty}a(r)p_k(r,s)=a(r)E_a(r,s)
 \le a(r)E_a(T,s)\in L_g^1.
\]
Dominated convergence therefore allows summation under the integral:
\begin{align*}
 \Id+\int_{[s,t)}A(r)\Phi(r,s)\dmu(r)
 &=\Id+\sum_{k=0}^{\infty}
   \int_{[s,t)}A(r)\Phi_k(r,s)\dmu(r)\\
 &=\sum_{k=0}^{\infty}\Phi_k(t,s)=\Phi(t,s).
\end{align*}
Thus \(\Phi\) satisfies \eqref{eq:fundamental-volterra}. The fundamental theorem of
Stieltjes calculus, in the form of Definition~\ref{def:acg}, now gives
\(\Phi(\cdot,s)\in\ACg\) and
\(\partial_g\Phi=A\Phi\) almost everywhere.

To prove uniqueness, let \(X,Y\) solve \eqref{eq:fundamental-volterra} and put
\(d(t)=\norm{X(t)-Y(t)}\). Then
\[
 d(t)\le\int_{[s,t)}a(r)d(r)\dmu(r).
\]
Since \(X-Y\) is bounded, let \(M=\sup_{[s,T]}d\). Repeated substitution yields
\[
 d(t)\le M p_k(t,s),\qquad k\ge1.
\]
Because \(\sum_kp_k(T,s)<\infty\), one has \(p_k(t,s)\to0\), and hence
\(d(t)=0\). This gives uniqueness without any additional regularity of \(A\).

For the cocycle identity, fix \(s\le r\). On \([r,T]\), the two functions
\[
 t\longmapsto\Phi(t,s),
 \qquad
 t\longmapsto\Phi(t,r)\Phi(r,s)
\]
satisfy the same Volterra equation and take the same value \(\Phi(r,s)\) at \(t=r\).
Uniqueness gives \eqref{eq:cocycle}. Finally, at
\(\tau\in[s,T)\cap\Dg\), Proposition~\ref{prop:canonical} applied to
\(\partial_g\Phi=A\Phi\) gives
\[
 \Phi(\tau^+,s)-\Phi(\tau,s)
 =\Deltag(\tau)A(\tau)\Phi(\tau,s),
\]
which is \eqref{eq:matrix-jump}. The uniform bound on the time triangle follows from
\[
 E_a(t,s)\le \exp\!\left(\int_{[t_0,T)}a(r)\dmu(r)\right).
\]
\end{proof}

\subsection{Regressivity and inverse evolution}

\begin{definition}\label{def:regressive}
The coefficient \(A\) is called \(g\)-regressive on \([t_0,T)\) if
\begin{equation}\label{eq:regressivity}
 \det\bigl(\Id+\Deltag(t)A(t)\bigr)\ne0,
 \qquad t\in[t_0,T)\cap\Dg.
\end{equation}
\end{definition}

Regressivity is not needed for forward existence; it is precisely the condition that makes each atomic transition invertible. Since every atom has positive
\(\mug\)-measure, the value of an \(L_g^1\)-coefficient at an atom is fixed by its
equivalence class; hence \eqref{eq:regressivity} is representative independent.

For a regressive \(A\), define on \([t_0,T)\) the regressive inverse coefficient
\begin{equation}\label{eq:circleminus}
 A^{\ominus}(t):=
 \begin{cases}
 -A(t),&t\notin\Dg,\\[1mm]
 -A(t)\bigl(\Id+\Deltag(t)A(t)\bigr)^{-1},&t\in\Dg.
 \end{cases}
\end{equation}

\begin{lemma}\label{lem:circleminus-L1}
If \(A\in\Lg{1}\) is \(g\)-regressive, then \(A^\ominus\in\Lg{1}\).
\end{lemma}

\begin{proof}
On the non-atomic part of \(\mug\), \(A^\ominus=-A\), so there is nothing to
prove. At an atom \(\tau\), set
\(\eta_\tau=\Deltag(\tau)\norm{A(\tau)}\). Since
\[
 \sum_{\tau\in\Dg\cap[t_0,T)}\eta_\tau
 =\int_{\Dg\cap[t_0,T)}\norm{A(t)}\dmu(t)
 \le\norm{A}_{L_g^1}<\infty,
\]
for each \(\varepsilon>0\) only finitely many atoms satisfy
\(\eta_\tau\ge\varepsilon\). Choose \(\varepsilon=1/2\). If
\(\eta_\tau<1/2\), the Neumann series gives
\[
 \norm{(I+\Deltag(\tau)A(\tau))^{-1}}
 \le\frac1{1-\eta_\tau}\le2,
\]
and hence
\[
 \Deltag(\tau)\norm{A^\ominus(\tau)}
 \le2\Deltag(\tau)\norm{A(\tau)}.
\]
The atoms with \(\eta_\tau\ge1/2\) form a finite set. At each of them regressivity
ensures that the inverse matrix exists and has finite norm, so their total contribution
to \(\int\norm{A^\ominus}\dmu\) is finite. Combining the non-atomic part, the small
atoms, and the finite exceptional set proves \(A^\ominus\in L_g^1\).
\end{proof}

\begin{theorem}[Inverse evolution]\label{thm:inverse}
If \(A\) is \(g\)-regressive, then \(\Phi(t,s)\) is invertible for all
\(t_0\le s\le t\le T\). Its inverse \(Z(t,s)=\Phi(t,s)^{-1}\) is the unique solution
of
\begin{equation}\label{eq:inverse-eq}
 Z'_g(t,s)=Z(t,s)A^\ominus(t),\qquad Z(s,s)=\Id,
\end{equation}
and satisfies
\begin{align}
 Z'_g(t,s)&=-Z(t^+,s)A(t),                     \label{eq:inverse-right}\\
 Z(t^+,s)&=
 Z(t,s)\bigl(\Id+\Deltag(t)A(t)\bigr)^{-1}.
                                                        \label{eq:inverse-jump}
\end{align}
Moreover,
\begin{equation}\label{eq:inverse-bound}
 \norm{\Phi(t,s)^{-1}}
 \le E_{\norm{A^\ominus}}(t,s)
 \le \exp\!\left(\int_{[s,t)}\norm{A^\ominus(r)}\dmu(r)\right).
\end{equation}
\end{theorem}

\begin{proof}
By Lemma~\ref{lem:circleminus-L1}, \(A^\ominus\in L_g^1\). Applying
Theorem~\ref{thm:evolution} to the transposed left-matrix equation
\[
 (Z^\top)'_g=(A^\ominus)^\top Z^\top,
 \qquad Z^\top(s,s)=I,
\]
produces a unique \(Z\in\ACg([s,T];\R^{n\times n})\) satisfying
\eqref{eq:inverse-eq}. At an atom \(t\), Proposition~\ref{prop:canonical} gives
\[
 Z(t^+,s)=Z(t,s)\bigl(I+\Deltag(t)A^\ominus(t)\bigr).
\]
The algebraic identity
\begin{equation}\label{eq:regressive-inverse-identity}
 I+\Deltag(t)A^\ominus(t)
 =\bigl(I+\Deltag(t)A(t)\bigr)^{-1}
\end{equation}
follows directly from \eqref{eq:circleminus}: at an atom,
\[
 I-\Deltag A(I+\Deltag A)^{-1}
 =[(I+\Deltag A)-\Deltag A](I+\Deltag A)^{-1}.
\]
This proves \eqref{eq:inverse-jump}. Moreover, because \(A\) commutes with every
polynomial and inverse polynomial in \(A\),
\[
 Z(t,s)A^\ominus(t)
 =-Z(t,s)A(t)(I+\Deltag(t)A(t))^{-1}
 =-Z(t^+,s)A(t)
\]
at an atom; off \(\Dg\) the same identity reduces to \(Z'_g=-ZA\). Hence
\eqref{eq:inverse-right} holds almost everywhere.

We now use the Stieltjes product rule. Since \(\Phi'_g=A\Phi\),
\[
 (Z\Phi)'_g
 =Z'_g\Phi+Z^+\Phi'_g
 =-Z^+A\Phi+Z^+A\Phi=0
 \quad\mug\text{-a.e.}
\]
The product belongs to \(\ACg\), so its integral representation implies that it is
constant. At \(t=s\) it equals \(I\); therefore
\[
 Z(t,s)\Phi(t,s)=I.
\]
A square matrix with a left inverse is invertible, and the left inverse is its inverse.
Thus \(Z(t,s)=\Phi(t,s)^{-1}\). Finally, applying the evolution estimate
\eqref{eq:evolution-bound} to the equation for \(Z^\top\) gives
\eqref{eq:inverse-bound}.
\end{proof}

For \(s<t\), we denote by \(\Phi(t,s^+)\) the propagation from the state immediately after a possible atom at \(s\):
\begin{equation}\label{eq:postjump}
 \Phi(t,s^+):=
 \begin{cases}
 \Phi(t,s)\bigl(\Id+\Deltag(s)A(s)\bigr)^{-1},&s\in\Dg,\\
 \Phi(t,s),&s\notin\Dg.
 \end{cases}
\end{equation}
Equivalently,
\[
 \Phi(t,s^+)
 =
 \Phi(t,t_0)\Phi(s^+,t_0)^{-1}.
\]

\begin{proposition}[Measurability of the propagators]\label{prop:propagator-measurable}
Assume that \(A\) is \(g\)-regressive. Then the maps
\[
 (t,s)\longmapsto \Phi(t,s),\qquad t_0\le s\le t\le T,
\]
and
\[
 (t,s)\longmapsto \Phi(t,s^+),\qquad t_0\le s<t\le T,
\]
are Borel measurable on their respective domains. In particular, for every fixed
\(t\), the maps \(s\mapsto\Phi(t,s)\) on \([t_0,t]\) and
\(s\mapsto\Phi(t,s^+)\) on \([t_0,t)\) are \(\mug\)-measurable and bounded.
\end{proposition}

\begin{proof}
Because the restriction of \(\mug\) to \([t_0,T)\) is the completion of a finite
Borel measure, the class of \(A\) admits a Borel representative. We choose such a representative, preserving the
uniquely determined values at atoms. The map \(t\mapsto\Deltag(t)\) is Borel
measurable (indeed regulated); see
\cite[Proposition~3.1]{FernandezMarquezTojoProduct2025}.

The function \(t\mapsto\Phi(t,t_0)\) belongs to \(\ACg\), and hence is
\(g\)-continuous and Borel measurable. For \(t<T\), its right-value representative is
\[
 \Phi(t^+,t_0)=(I+\Deltag(t)A(t))\Phi(t,t_0),
\]
so it is also Borel measurable. By Theorem~\ref{thm:inverse}, both matrices are
invertible. Since inversion is continuous on \(\operatorname{GL}(n,\R)\), the maps
\[
 t\mapsto\Phi(t,t_0)^{-1},
 \qquad
 t\mapsto\Phi(t^+,t_0)^{-1}
\]
are Borel measurable on \([t_0,T]\) and \([t_0,T)\), respectively.

For \(t_0\le s\le t\le T\), the cocycle identity gives
\[
 \Phi(t,s)=\Phi(t,t_0)\Phi(s,t_0)^{-1},
\]
whereas, for \(s<t\), the definition of post-jump propagation gives
\[
 \Phi(t,s^+)=\Phi(t,t_0)\Phi(s^+,t_0)^{-1}.
\]
Both are products of functions measurable in separate variables, and hence jointly
Borel measurable on their respective domains. Their boundedness follows from
\eqref{eq:evolution-bound} and \eqref{eq:inverse-bound}. To make the bound for the
right-value inverse explicit, write \(Z(r)=\Phi(r,t_0)^{-1}\). Since
\(Z\in\ACg(I)\), Proposition~\ref{prop:canonical} and the integral representation give
\[
 \sup_{r\in[t_0,T)}\norm{Z(r^+)}
 \le \norm{I}+\int_{[t_0,T)}\norm{Z'_g(\rho)}\,d\mug(\rho)<\infty.
\]
Consequently,
\[
 \norm{\Phi(t,s^+)}
 \le \sup_{r\in I}\norm{\Phi(r,t_0)}
      \sup_{r\in[t_0,T)}\norm{\Phi(r^+,t_0)^{-1}}<\infty.
\]
\end{proof}

\subsection{Constant coefficients}

Assume now that \(A(t)\equiv A\in\R^{n\times n}\).

\begin{theorem}[Matrix \(g\)-exponential]\label{thm:constant}
For \(t_0\le s\le t\le T\),
\begin{equation}\label{eq:constant-formula}
 \Phi_A(t,s)
 =
 \exp\!\bigl(A\mugc([s,t))\bigr)
 \prod_{\tau\in[s,t)\cap\Dg}
 \bigl(\Id+\Deltag(\tau)A\bigr).
\end{equation}
The product converges in matrix norm and its value is independent of the order of the
factors. If \(A\) is \(g\)-regressive, then
\begin{equation}\label{eq:constant-inverse}
 \Phi_A(t,s)^{-1}
 =
 \left[
 \prod_{\tau\in[s,t)\cap\Dg}
 \bigl(\Id+\Deltag(\tau)A\bigr)^{-1}
 \right]
 \exp\!\bigl(-A\mugc([s,t))\bigr).
\end{equation}
\end{theorem}

\begin{proof}
Let \(J=[s,t)\cap\Dg\). For a finite subset \(F\subset J\), write
\[
 P_F:=\prod_{\tau\in F}(I+\Deltag(\tau)A).
\]
All factors are polynomials in the same matrix \(A\), and therefore commute. Moreover,
\[
 \sum_{\tau\in J}\norm{\Deltag(\tau)A}
 \le\norm{A}\mug([s,t))<\infty.
\]
For finite sets \(F\subset G\subset J\), submultiplicativity and
\(1+x\le e^x\) give
\[
 \norm{P_F}\le
 \prod_{\tau\in F}\bigl(1+\norm{\Deltag(\tau)A}\bigr)
 \le \exp\!\left(\sum_{\tau\in J}\norm{\Deltag(\tau)A}\right)
 =:C_A,
\]
and
\[
 \norm{P_G-P_F}
 \le C_A\left[
 \exp\!\left(\sum_{\tau\in G\setminus F}
          \norm{\Deltag(\tau)A}\right)-1\right].
\]
The last quantity tends to zero when \(F\) contains a sufficiently large finite part
of the summable family. Hence \((P_F)\) is a Cauchy net in the finite-dimensional
Banach algebra \(\R^{n\times n}\), and therefore converges. Commutativity shows that
the limit is independent of the enumeration.

To identify the limit with the fundamental matrix, enumerate the atoms of \([s,T)\)
as \(\{\tau_k\}_{k\ge1}\) and introduce the truncated measures
\[
 \mug^{(N)}=\mugc+\sum_{k=1}^N\Deltag(\tau_k)\delta_{\tau_k}.
\]
Set
\[
 F_N(t)=
 \exp\!\bigl(A\mugc([s,t))\bigr)
 \prod_{\substack{k\le N\\ \tau_k<t}}
 (I+\Deltag(\tau_k)A).
\]
For completeness, put \(\theta(t)=\mugc([s,t))\). Since \(\mugc\) is
non-atomic, \(\theta\) is continuous and \(d\mu_\theta=d\mugc\). By the
Lebesgue--Stieltjes substitution formula,
\begin{align*}
 e^{A\theta(t)}-I
 &=\int_0^{\theta(t)}Ae^{Ar}\,dr
  =\int_{[s,t)}Ae^{A\theta(r)}\dmuc(r).
\end{align*}
Thus the first factor is exactly the fundamental matrix driven by \(\mugc\), including
when that measure is singular continuous. At each included atom \(\tau_k\), the
second factor gives
\[
 F_N(\tau_k^+)=(I+\Deltag(\tau_k)A)F_N(\tau_k).
\]
Hence the fundamental theorem yields
\begin{equation}\label{eq:constant-truncated-volterra}
 F_N(t)=I+\int_{[s,t)}AF_N(r)\,d\mug^{(N)}(r).
\end{equation}

The product criterion gives uniform convergence \(F_N\to F\) on \([s,T]\). A
uniform bound is supplied by
\[
 \norm{F_N(t)}
 \le\exp\!\bigl(\norm{A}\mug([s,T))\bigr)=:M_A.
\]
The total variation norm below is taken after restriction to \([s,T)\). Since \(A\) is constant,
\begin{align*}
 &\norm{\int_{[s,t)}AF_N\,d\mug^{(N)}
       -\int_{[s,t)}AF\,d\mug}\\
 &\quad\le
 \norm{A}\mug([s,T))\norm{F_N-F}_\infty
 +\norm{A}M_A\,\norm{\mug-\mug^{(N)}}_{\mathrm{TV}},
\end{align*}
where
\[
 \norm{\mug-\mug^{(N)}}_{\mathrm{TV}}
 =\sum_{k>N}\Deltag(\tau_k)\longrightarrow0.
\]
Passing to the limit in \eqref{eq:constant-truncated-volterra} shows that \(F\)
solves \eqref{eq:fundamental-volterra}. Uniqueness from
Theorem~\ref{thm:evolution} proves \eqref{eq:constant-formula}.

If \(A\) is regressive, every jump factor is invertible. The continuous exponential
is invertible with inverse \(\exp(-A\mugc([s,t)))\), and all factors commute. Inverting
the product therefore gives \eqref{eq:constant-inverse}.
\end{proof}

\begin{remark}[On the product representation]\label{rem:no-logarithm}
Formula \eqref{eq:constant-formula} requires neither diagonalizability nor a matrix
logarithm. In particular, no positivity condition on the eigenvalues of the jump
factors is needed for forward evolution. A logarithmic representation would require a
compatible choice of logarithm for every factor
\(\Id+\Deltag(\tau)A\), and it would still have to retain the dependence on each
individual atom. Accordingly, the product representation is the natural form of the fundamental matrix in the presence of atoms.
\end{remark}

\begin{corollary}\label{cor:diagonal}
If \(A=P\operatorname{diag}(\lambda_1,\ldots,\lambda_n)P^{-1}\) over
\(\C\), then
\[
 \Phi_A(t,s)
 =
 P\operatorname{diag}
 \bigl(\eG(\lambda_1;t,s),\ldots,\eG(\lambda_n;t,s)\bigr)P^{-1},
\]
where
\[
 \eG(\lambda;t,s)
 :=
 \exp\!\bigl(\lambda\mugc([s,t))\bigr)
 \prod_{\tau\in[s,t)\cap\Dg}
 \bigl(1+\lambda\Deltag(\tau)\bigr).
\]
For a real matrix, complex conjugate components combine to give a real fundamental
matrix.
\end{corollary}

\begin{proposition}[Forward variation of constants without regressivity]
\label{prop:constant-forward-voc}
Let \(A\in\R^{n\times n}\) be arbitrary; no regressivity assumption is imposed. For
\(s<t\), define the forward post-jump propagator by
\begin{equation}\label{eq:constant-forward-postjump}
 \Phi_A^{\to}(t,s^+)
 :=\exp\!\bigl(A\mugc([s,t))\bigr)
   \prod_{\tau\in(s,t)\cap\Dg}
   \bigl(I+\Deltag(\tau)A\bigr).
\end{equation}
Then the unique forward solution of
\[
 x'_g=Ax+f,\qquad x(t_0)=x_0,
\]
with \(f\in L_g^1\), is
\begin{equation}\label{eq:constant-forward-voc}
 x(t)=\Phi_A(t,t_0)x_0+
 \int_{[t_0,t)}\Phi_A^{\to}(t,s^+)f(s)\dmu(s).
\end{equation}
If \(A\) is \(g\)-regressive, \(\Phi_A^{\to}(t,s^+)=\Phi_A(t,s^+)\), so
\eqref{eq:constant-forward-voc} agrees with Theorem~\ref{thm:variation}.
\end{proposition}

\begin{proof}
The product in \eqref{eq:constant-forward-postjump} contains precisely the jumps
strictly after \(s\); hence it propagates a state already located after a possible jump
at \(s\). The same summability argument as in Theorem~\ref{thm:constant} shows that
it is well defined, measurable, and bounded on the time triangle. For fixed \(s\), it
satisfies the forward integral identity
\begin{equation}\label{eq:forward-kernel-integral}
 \Phi_A^{\to}(t,s^+)
 =I+\int_{(s,t)}A\Phi_A^{\to}(r,s^+)\dmu(r).
\end{equation}
Indeed, the interval starts immediately after \(s\), so the possible atom at \(s\) is
excluded, whereas every later atom is included.

Let
\[
 J(t):=\int_{[t_0,t)}\Phi_A^{\to}(t,s^+)f(s)\dmu(s).
\]
Substituting \eqref{eq:forward-kernel-integral} gives
\begin{align*}
 J(t)
 &=\int_{[t_0,t)}f(s)\dmu(s)\\
 &\quad+\int_{[t_0,t)}\int_{(s,t)}
 A\Phi_A^{\to}(r,s^+)f(s)\dmu(r)\dmu(s).
\end{align*}
The integrand is absolutely integrable because the propagator is bounded and
\(f\in L_g^1\). Fubini's theorem therefore permits exchanging the order of
integration over \(\{(s,r):t_0\le s<r<t\}\), and yields
\begin{align*}
 J(t)
 &=\int_{[t_0,t)}f(s)\dmu(s)
 +\int_{[t_0,t)}A
   \left[\int_{[t_0,r)}\Phi_A^{\to}(r,s^+)f(s)\dmu(s)\right]\dmu(r)\\
 &=\int_{[t_0,t)}[f(r)+AJ(r)]\dmu(r).
\end{align*}
The homogeneous term \(Y(t)=\Phi_A(t,t_0)x_0\) satisfies
\(Y(t)=x_0+\int_{[t_0,t)}AY(r)\dmu(r)\). Thus \(X=Y+J\), the right-hand side of
\eqref{eq:constant-forward-voc}, satisfies
\[
 X(t)=x_0+\int_{[t_0,t)}[AX(r)+f(r)]\dmu(r).
\]
This is precisely the canonical integral formulation. Forward uniqueness follows from
the Volterra argument in Theorem~\ref{thm:evolution}, which does not require
regressivity.

If \(A\) is regressive, formula \eqref{eq:constant-formula} shows that
\(\Phi_A(t,s)\) contains the factor \(I+\Deltag(s)A\) exactly when \(s\) is an
atom. Multiplication by its inverse in \eqref{eq:postjump} deletes that factor and
leaves exactly \eqref{eq:constant-forward-postjump}. This proves the last assertion.
\end{proof}

\begin{counterexample}[Dependence on the complete jump history]
\label{ex:wrong-exponential}
Let \(n=1\), \(A=1\), \(t_0=0\), and
\[
 g(t)=
 \begin{cases}
 t,&t\le1,\\
 t+1,&t>1.
 \end{cases}
\]
Then \(\mug\) is Lebesgue measure plus one atom of mass \(1\) at \(t=1\).
The solution of \(x'_g=x\), \(x(0)=1\), is
\[
 x(t)=
 \begin{cases}
 e^t,&t\le1,\\
 2e^t,&t>1.
 \end{cases}
\]
In contrast, the expression \(e^{g(t)-g(0)}\) gives \(e^{t+1}\) for \(t>1\).
Thus the fundamental solution is not determined, in general, by the total increment \(g(t)-g(s)\) alone; the individual jump factors must be retained. Notice also that
for every \(t>1\) the endpoint \(t\) is not itself an atom, although the preceding atom
still changes the value of the solution. Thus a formula that distinguishes cases only
according to whether the terminal point belongs to \(\Dg\) cannot encode the evolution
history.
\end{counterexample}

\subsection{Variation of constants}

\begin{theorem}[Variation-of-constants formula]\label{thm:variation}
Assume that \(A\in\Lg{1}\) is \(g\)-regressive and
\(f\in\Lg{1}([t_0,T);\R^n)\). Then
\begin{equation}\label{eq:linear-IVP}
 \begin{cases}
 x'_g(t)=A(t)x(t)+f(t),&\mug\text{-a.e. }t\in[t_0,T),\\
 x(t_0)=x_0,
 \end{cases}
\end{equation}
has a unique solution \(x\in\ACg(I;\R^n)\), and
\begin{equation}\label{eq:variation}
 x(t)
 =
 \Phi(t,t_0)x_0+
 \int_{[t_0,t)}\Phi(t,s^+)f(s)\dmu(s).
\end{equation}
Equivalently,
\begin{equation}\label{eq:variation-reference}
 x(t)
 =
 \Phi(t,t_0)
 \left[
 x_0+
 \int_{[t_0,t)}
 \Phi(s^+,t_0)^{-1}f(s)\dmu(s)
 \right].
\end{equation}
\end{theorem}

\begin{proof}
Define
\[
 b(t):=\Phi(t^+,t_0)^{-1}f(t).
\]
By Proposition~\ref{prop:propagator-measurable}, \(b\) is measurable. The inverse
bound \eqref{eq:inverse-bound} gives
\[
 \norm{b(t)}
 \le C_\Phi\norm{f(t)},
 \qquad
 C_\Phi:=\sup_{t\in[t_0,T)}\norm{\Phi(t^+,t_0)^{-1}}<\infty,
\]
so \(b\in L_g^1\). Let
\[
 c(t):=x_0+\int_{[t_0,t)}b(s)\dmu(s).
\]
Then \(c\in\ACg\), \(c(t_0)=x_0\), and \(c'_g=b\). Set
\(x(t)=\Phi(t,t_0)c(t)\). By Proposition~\ref{prop:product},
\begin{align*}
 x'_g(t)
 &=A(t)\Phi(t,t_0)c(t)+\Phi(t^+,t_0)c'_g(t)\\
 &=A(t)x(t)+\Phi(t^+,t_0)\Phi(t^+,t_0)^{-1}f(t)\\
 &=A(t)x(t)+f(t)
\end{align*}
for almost every \(t\), and \(x(t_0)=x_0\). Substitution of the definition of \(c\)
gives \eqref{eq:variation-reference}.

For \(s<t\), the cocycle identity, interpreted from the state immediately after
\(s\), yields
\[
 \Phi(t,t_0)\Phi(s^+,t_0)^{-1}=\Phi(t,s^+).
\]
Multiplying the integral in \eqref{eq:variation-reference} by
\(\Phi(t,t_0)\) and using this identity gives \eqref{eq:variation}.

If \(x_1,x_2\) are two solutions, their difference \(y=x_1-x_2\) satisfies
\[
 y(t)=\int_{[t_0,t)}A(s)y(s)\dmu(s),\qquad y(t_0)=0.
\]
The uniqueness argument in Theorem~\ref{thm:evolution}, applied columnwise (or directly
to this vector equation), gives \(y=0\). Thus the solution is unique.
\end{proof}

\begin{corollary}[Nonhomogeneous jump]\label{cor:forcing-jump}
At every \(t\in[t_0,T)\cap\Dg\), the solution in
Theorem~\ref{thm:variation} satisfies
\begin{equation}\label{eq:forcing-jump}
 x(t^+)
 =
 \bigl(\Id+\Deltag(t)A(t)\bigr)x(t)
 +
 \Deltag(t)f(t).
\end{equation}
\end{corollary}

The forcing term in \eqref{eq:forcing-jump} is not multiplied by the homogeneous jump factor. This accounts for the post-jump argument \(s^+\) in \eqref{eq:variation}.

\begin{example}[Necessity of the post-jump propagator]\label{ex:prejump-kernel}
Consider the scalar problem on an interval containing a single atom \(\tau\) of mass
\(\delta>0\), with \(A(t)\equiv a\), \(f(t)\equiv b\), and \(x(\tau)=0\). The
differential equation gives directly
\[
 x(\tau^+)=\delta b.
\]
The post-jump formula reproduces this value because
\(\Phi(\tau^+,\tau^+)=1\). If one were to use \(\Phi(t,s)\) in place of
\(\Phi(t,s^+)\), the atomic contribution at \(t=\tau^+\) would instead be
\[
 \Phi(\tau^+,\tau)b\,\delta=(1+a\delta)b\delta,
\]
which contains an extraneous homogeneous jump factor. Hence the use of the pre-jump kernel introduces an additional factor and does not reproduce the prescribed jump relation.
\end{example}

\section{Exact controllability}\label{sec:control}

Fix \(t_1\in(t_0,T]\) and consider
\begin{equation}\label{eq:linear-control}
 \begin{cases}
 x'_g(t)=A(t)x(t)+B(t)u(t),
   &\mug\text{-a.e. }t\in[t_0,t_1),\\
 x(t_0)=x_0.
 \end{cases}
\end{equation}
Assume that \(A\in\Lg{1}\) is \(g\)-regressive,
\[
 B\in\Lg{2}([t_0,t_1);\R^{n\times m}),
 \qquad
 u\in\cU:=\Lg{2}([t_0,t_1);\R^m).
\]
We equip \(\cU\) with the Hilbert inner product
\begin{equation}\label{eq:control-inner-product}
 \ip{u}{v}_{\cU}:=
 \int_{[t_0,t_1)}u(s)^\top v(s)\dmu(s).
\end{equation}
Since \(\mug([t_0,t_1))<\infty\), H\"older's inequality gives
\[
 \norm{Bu}_{L_g^1}
 \le \norm{B}_{L_g^2}\norm{u}_{L_g^2},
\]
so the nonhomogeneous term belongs to \(L_g^1\).

\subsection{Endpoint operator and Gramian}

By Theorem~\ref{thm:variation},
\[
 x(t_1)=\Phi(t_1,t_0)x_0+\cK u,
\]
where
\begin{equation}\label{eq:endpoint}
 \cK u
 :=
 \int_{[t_0,t_1)}H(s)u(s)\dmu(s),
 \qquad
 H(s):=\Phi(t_1,s^+)B(s).
\end{equation}
By Proposition~\ref{prop:propagator-measurable}, \(H\) is measurable. If
\[
 M_{t_1}:=\sup_{s\in[t_0,t_1)}\norm{\Phi(t_1,s^+)}<\infty,
\]
then Cauchy--Schwarz gives
\[
 \norm{\cK u}
 \le M_{t_1}\norm{B}_{L_g^2}\norm{u}_{L_g^2}.
\]
Hence \(\cK:\cU\to\R^n\) is bounded.

\begin{definition}
System \eqref{eq:linear-control} is \emph{exactly controllable} on
\([t_0,t_1]\) if, for every \(x_0,x_1\in\R^n\), there exists \(u\in\cU\)
such that the associated solution satisfies \(x(t_1)=x_1\).
\end{definition}

\begin{definition}
The controllability Gramian is
\begin{equation}\label{eq:W}
 W(t_0,t_1)
 :=
 \int_{[t_0,t_1)}
 H(s)H(s)^\top\dmu(s).
\end{equation}
\end{definition}

Its continuous--atomic decomposition is
\begin{align}
 W(t_0,t_1)
 ={}&
 \int_{[t_0,t_1)}
 H(s)H(s)^\top\dmuc(s) \notag\\
 &+
 \sum_{\tau\in[t_0,t_1)\cap\Dg}
 \Deltag(\tau)H(\tau)H(\tau)^\top.
 \label{eq:W-decomp}
\end{align}

\begin{example}[Role of the atomic mass]\label{ex:atomic-weight}
Let \(n=m=1\), \(A=0\), \(B=1\), and suppose that the control interval contains
one atom \(\tau\) of mass \(\delta>0\) and no other \(\mug\)-mass. Then
\[
 x(t_1)=x_0+\delta u(\tau),\qquad
 \norm{u}_{L_g^2}^2=\delta\abs{u(\tau)}^2,
\]
and the Gramian is \(W=\delta\). To produce an endpoint displacement \(d\), the
minimum-energy control is \(u(\tau)=d/\delta\), with energy \(d^2/\delta\). Omitting
the factor \(\delta\) from the Gramian would give the control \(u(\tau)=d\), whose
endpoint displacement is only \(\delta d\) unless \(\delta=1\). Therefore the factors \(\Deltag(\tau)\) in \eqref{eq:W-decomp} are fixed simultaneously by the endpoint map and by the \(L_g^2\)-norm.
\end{example}

\begin{theorem}[Gramian criterion]\label{thm:gramian}
The following statements are equivalent:
\begin{enumerate}[label=\textnormal{(\roman*)}]
\item system \eqref{eq:linear-control} is exactly controllable;
\item \(\cK\) is surjective;
\item \(W(t_0,t_1)\) is positive definite;
\item there is no nonzero \(z\in\R^n\) such that
\[
 B(s)^\top\Phi(t_1,s^+)^\top z=0
 \quad\text{for \(\mug\)-almost every \(s\in[t_0,t_1)\)}.
\]
\end{enumerate}
\end{theorem}

\begin{proof}
For \(u\in\cU\) and \(z\in\R^n\), linearity of the Bochner integral gives
\begin{align*}
 \ip{\cK u}{z}_{\R^n}
 &=z^\top\int_{[t_0,t_1)}H(s)u(s)\dmu(s)\\
 &=\int_{[t_0,t_1)}u(s)^\top H(s)^\top z\dmu(s)
 =\ip{u}{H(\cdot)^\top z}_{\cU}.
\end{align*}
Thus
\[
 (\cK^*z)(s)=H(s)^\top z.
\]
The function belongs to \(\cU\) because \(H\in L_g^2\). Applying \(\cK\) once
more,
\[
 \cK\cK^*z
 =\int_{[t_0,t_1)}H(s)H(s)^\top z\dmu(s)
 =W(t_0,t_1)z.
\]
Consequently,
\begin{equation}\label{eq:gramian-quadratic-form}
 z^\top W(t_0,t_1)z
 =\norm{\cK^*z}_{\cU}^2
 =\int_{[t_0,t_1)}\norm{H(s)^\top z}^2\dmu(s).
\end{equation}

The system is exactly controllable if and only if \(\Range\cK=\R^n\), because the
required displacement
\(d=x_1-\Phi(t_1,t_0)x_0\) ranges over all of \(\R^n\). The Hilbert-space identity
\((\Range\cK)^\perp=\Ker\cK^*\), together with the finite-dimensionality of the
codomain, shows that \(\cK\) is surjective if and only if \(\cK^*\) is injective.
By \eqref{eq:gramian-quadratic-form}, injectivity of \(\cK^*\) is equivalent to
\(z^\top Wz>0\) for every \(z\ne0\), i.e.\ to positive definiteness of \(W\). The same
identity shows that \(\cK^*z=0\) exactly when
\(B(s)^\top\Phi(t_1,s^+)^\top z=0\) almost everywhere. This proves all equivalences.
\end{proof}

\begin{corollary}[Minimum-energy control]\label{cor:minenergy}
Assume that \(W=W(t_0,t_1)\) is invertible and put
\[
 d:=x_1-\Phi(t_1,t_0)x_0.
\]
Then
\begin{equation}\label{eq:minenergy}
 u^\ast(s)
 =
 B(s)^\top\Phi(t_1,s^+)^\top W^{-1}d
\end{equation}
steers \(x_0\) to \(x_1\), and it is the unique control of minimum
\(\Lg{2}\)-norm among all controls with this endpoint. Its minimum energy is
\begin{equation}\label{eq:min-energy-value}
 \norm{u^*}_{L_g^2}^2=d^\top W^{-1}d.
\end{equation}
\end{corollary}

\begin{proof}
In operator notation, \(u^*=\cK^*W^{-1}d\). Since \(W=\cK\cK^*\),
\[
 \cK u^*=\cK\cK^*W^{-1}d=d,
\]
so \(u^*\) produces the prescribed endpoint. If \(u\) is any other admissible
control, then \(\cK(u-u^*)=0\), and hence
\[
 u=u^*+v,\qquad v\in\Ker\cK.
\]
On the other hand, \(u^*\in\Range\cK^*\), and
\(\Range\cK^*=(\Ker\cK)^\perp\). Therefore
\[
 \norm{u}_{\cU}^2
 =\norm{u^*}_{\cU}^2+\norm{v}_{\cU}^2
 \ge\norm{u^*}_{\cU}^2,
\]
with equality only when \(v=0\). Finally,
\begin{align*}
 \norm{u^*}_{\cU}^2
 &=\ip{\cK^*W^{-1}d}{\cK^*W^{-1}d}_{\cU}\\
 &=\ip{W^{-1}d}{\cK\cK^*W^{-1}d}_{\R^n}
 =d^\top W^{-1}d.
\end{align*}
\end{proof}

\begin{remark}[Consistency with the Pontryagin principle on time scales]
\label{rem:pmp-timescales}
Under the regularity hypotheses of \citet{BourdinTrelat2013}, and provided that the
minimum-energy control belongs to the admissible class considered there, the
Pontryagin principle on a time scale provides a consistency check for
Corollary~\ref{cor:minenergy}. Consider the quadratic cost
\[
 J(u)=\frac12\int_{[t_0,t_1)}\norm{u(s)}^2\,d\mu_g(s)
\]
with unconstrained controls. With the normal multiplier, the Hamiltonian is
\[
 \mathcal H(x,u,p,s)=p^\top(A(s)x+B(s)u)-\frac12u^\top u.
\]
At right-dense points Hamiltonian maximization gives
\(u^*(s)=B(s)^\top p(s)\). At right-scattered points, because the control set is
\(\R^m\), the directional condition yields the corresponding stationarity relation
with the shifted adjoint,
\[
 u^*(s)=B(s)^\top p^\sigma(s)=B(s)^\top p(s^+).
\]
For a terminal multiplier \(\lambda\), adjoint propagation gives
\[
 p(s^+)=\Phi(t_1,s^+)^\top\lambda.
\]
Choosing \(\lambda=W^{-1}d\) yields the same candidate as
\eqref{eq:minenergy}; strict convexity of the cost on the affine set of feasible
controls identifies it with the unique minimum-energy control. Thus the
Hilbert-space endpoint-operator proof used here is consistent with the time-scale
Pontryagin formalism, while remaining directly applicable to general Stieltjes
measures, including singular-continuous components that are not generated by a time
scale.
\end{remark}

\subsection{Reachable subspace}

For a measurable matrix function \(H:[t_0,t_1)\to\R^{n\times m}\), define
\begin{equation}\label{eq:essspan-def}
 \essspan\Range H
 :=
 \left\{
 z\in\R^n:H(s)^\top z=0\quad\mug\text{-a.e.}
 \right\}^{\perp}.
\end{equation}
Equivalently, this is the smallest subspace \(V\subset\R^n\) such that
\(\Range H(s)\subset V\) for \(\mug\)-almost every \(s\).

\begin{proposition}\label{prop:reachable}
The reachable subspace from the origin at \(t_1\) is
\begin{equation}\label{eq:reachable}
 \cR_g(t_0,t_1)
 :=
 \Range\cK
 =
 \essspan
 \left\{
 \Range\bigl(\Phi(t_1,s^+)B(s)\bigr):
 s\in[t_0,t_1)
 \right\}.
\end{equation}
\end{proposition}

\begin{proof}
Let
\[
 V:=\left\{z\in\R^n:H(s)^\top z=0\ \text{for }\mug\text{-a.e. }s\right\}^\perp.
\]
For \(z\in\R^n\),
\[
 z\in(\Range\cK)^\perp
 \quad\Longleftrightarrow\quad
 \ip{\cK u}{z}=0\ \text{for every }u\in\cU
 \quad\Longleftrightarrow\quad
 \cK^*z=0.
\]
Since \((\cK^*z)(s)=H(s)^\top z\), the last condition is exactly
\(z\in V^\perp\). Hence
\[
 (\Range\cK)^\perp=V^\perp.
\]
Both \(\Range\cK\) and \(V\) are subspaces of the finite-dimensional space
\(\R^n\), so taking orthogonal complements again gives
\(\Range\cK=V\), which is \eqref{eq:reachable}.

To see the stated equivalent description, let \(V_0\) be the smallest subspace with
\(\Range H(s)\subset V_0\) almost everywhere. Then \(z\in V_0^\perp\) if and only
if \(H(s)^\top z=0\) almost everywhere. Therefore \(V_0=V\).
\end{proof}

\subsection{Constant coefficients and the Kalman condition}

Assume \(A(t)\equiv A\) and \(B(t)\equiv B\), and set
\[
 \cC(A,B)
 :=
 \begin{bmatrix}
 B&AB&A^2B&\cdots&A^{n-1}B
 \end{bmatrix}.
\]

\begin{proposition}[Kalman necessity for arbitrary derivators]\label{prop:kalman-necessary}
Let \(A\) and \(B\) be constant. No regressivity assumption is required. Let
\(\cR_g(t_0,t_1)\) denote the range of the forward endpoint map defined by
Proposition~\ref{prop:constant-forward-voc}. Then, for every derivator \(g\),
\[
 \cR_g(t_0,t_1)\subset\Range\cC(A,B).
\]
Consequently, exact controllability implies
\[
 \rank\cC(A,B)=n.
\]
\end{proposition}

\begin{proof}
Let
\[
 \mathcal A(A):=\Span\{I,A,\ldots,A^{n-1}\}.
\]
By the Cayley--Hamilton theorem, \(\mathcal A(A)\) is a unital subalgebra of
\(\R^{n\times n}\): every power \(A^k\), \(k\ge n\), is a linear combination of
lower powers. Being finite-dimensional, \(\mathcal A(A)\) is closed.

For every \(s<t_1\), the continuous exponential
\(\exp(A\mugc([s,t_1)))\) is the norm limit of polynomials in \(A\), and every
finite partial product of the jump factors
\(I+\Deltag(\tau)A\) is a polynomial in \(A\). Formula
\eqref{eq:constant-forward-postjump} and closedness of \(\mathcal A(A)\) therefore
imply
\[
 \Phi_A^{\to}(t_1,s^+)\in\mathcal A(A).
\]
Thus there are scalars \(\alpha_0(s),\ldots,\alpha_{n-1}(s)\) such that
\[
 \Phi_A^{\to}(t_1,s^+)B
 =\sum_{k=0}^{n-1}\alpha_k(s)A^kB.
\]
Every column of the left-hand side belongs to
\(\Range[B,AB,\ldots,A^{n-1}B]\). By
Proposition~\ref{prop:constant-forward-voc}, the endpoint operator for constant
coefficients is obtained by integrating these forward propagated directions.
Therefore
\[
 \cR_g(t_0,t_1)\subset\Range\cC(A,B).
\]
If the system is exactly controllable, the left-hand side is \(\R^n\), and therefore
the Kalman matrix has rank \(n\).
\end{proof}

\begin{counterexample}[The Kalman condition is not sufficient for an arbitrary derivator]
\label{ex:kalman-not-sufficient}
Let \(n=2\), \(m=1\),
\[
 A=
 \begin{pmatrix}
 0&1\\
 0&0
 \end{pmatrix},
 \qquad
 B=
 \begin{pmatrix}
 0\\1
 \end{pmatrix},
\]
and let
\[
 g(t)=
 \begin{cases}
 0,&t\le1,\\
 1,&t>1.
 \end{cases}
\]
On \([0,2]\), \(\mug=\delta_1\). Moreover,
\[
 \cC(A,B)=
 \begin{pmatrix}
 0&1\\
 1&0
 \end{pmatrix},
\]
so the Kalman matrix has rank \(2\), and
\(\Id+A\) is invertible. Nevertheless, for \(x(0)=0\),
\[
 x(2)=B\,u(1),
\]
because the control acts only at the single atom and no evolution occurs afterwards.
The reachable subspace is \(\Range B\), of dimension \(1\). Thus the system is not
controllable.
\end{counterexample}

Thus a rank condition involving only \(A\) and \(B\) cannot characterize controllability without an additional hypothesis on the temporal structure induced by \(g\).

\begin{lemma}[Stieltjes change of variables on a continuous component]
\label{lem:continuous-pushforward}
Let \(g\) be continuous on \([a,c]\). Then, for every
nonnegative Borel function \(\psi\) (and hence for every integrable Borel function),
\begin{equation}\label{eq:continuous-pushforward}
 \int_{[a,c)}\psi(g(s))\dmu(s)
 =\int_{[g(a),g(c))}\psi(r)\,dr.
\end{equation}
Equivalently, the pushforward of \(\mug|_{[a,c)}\) by \(g\) is Lebesgue measure on
\([g(a),g(c))\).
\end{lemma}

\begin{proof}
Apply Lemma~\ref{lem:continuous-cumulative-substitution} to the restriction of \(g\)
to \([a,c]\). Its associated Stieltjes measure is precisely
\(\mug|_{[a,c)}\), because both measures take the value
\(g(\beta)-g(\alpha)\) on every half-open interval
\([\alpha,\beta)\subset[a,c)\). Formula
\eqref{eq:continuous-pushforward} follows directly.
\end{proof}

\begin{theorem}[Kalman criterion under a nontrivial continuous component]\label{thm:continuous-window}
Assume that \(A\) and \(B\) are constant, \(A\) is \(g\)-regressive, and there is a
nondegenerate interval \([a,b]\subset[t_0,t_1)\) on which \(g\) is continuous and
nonconstant, equivalently \(g(b)>g(a)\). Then
\[
 \text{\eqref{eq:linear-control} is exactly controllable}
 \quad\Longleftrightarrow\quad
 \rank\cC(A,B)=n.
\]
\end{theorem}

\begin{proof}
Necessity is Proposition~\ref{prop:kalman-necessary}. Assume now that
\(\rank\cC(A,B)=n\). By Theorem~\ref{thm:gramian}, it is enough to prove that
\[
 B^\top\Phi_A(t_1,s^+)^\top z=0
 \quad\mug\text{-a.e. }s\in[t_0,t_1)
\]
implies \(z=0\). Since \(g\) is continuous and nonconstant on \([a,b]\), monotonicity gives
\(g(b)>g(a)\), and continuity allows us to choose \(c\in(a,b)\) such that
\(g(c)>g(a)\). There are no atoms in \([a,c]\), so \(s^+=s\) there. Restricting the
annihilation identity to
\([a,c)\) gives
\begin{equation}\label{eq:continuous-annihilation}
 B^\top\Phi_A(t_1,s)^\top z=0
 \quad\mug\text{-a.e. }s\in[a,c).
\end{equation}

By the cocycle identity and the constant-coefficient formula,
\[
 \Phi_A(t_1,s)
 =\Phi_A(t_1,c)\Phi_A(c,s)
 =\Phi_A(t_1,c)\exp(A[g(c)-g(s)]).
\]
Set \(w=\Phi_A(t_1,c)^\top z\). Since \(A\) is regressive,
\(\Phi_A(t_1,c)\) is invertible. Equation \eqref{eq:continuous-annihilation} becomes
\[
 B^\top\exp(A^\top[g(c)-g(s)])w=0
 \quad\mug\text{-a.e. }s\in[a,c).
\]
Lemma~\ref{lem:continuous-pushforward}, followed by the reflection
\(\theta=g(c)-r\), shows that
\[
 F(\theta):=B^\top e^{A^\top\theta}w=0
\]
for Lebesgue-almost every
\(\theta\in(0,g(c)-g(a)]\). Every component of \(F\) is an entire real-analytic function. By the identity
theorem for real-analytic functions, a component that vanishes on a set with an
accumulation point vanishes identically; see, e.g.,
\citep{KrantzParks2002}. Since a full-measure subset of a nondegenerate
interval has accumulation points, the almost-everywhere identity implies
\(F\equiv0\). Differentiating at
\(\theta=0\) gives
\[
 B^\top(A^\top)^kw=0,
 \qquad k=0,1,\ldots,n-1.
\]
Equivalently,
\[
 \cC(A,B)^\top w=0.
\]
The Kalman rank condition implies \(w=0\). Invertibility of
\(\Phi_A(t_1,c)^\top\) then yields \(z=0\). The Gramian criterion completes the
proof.
\end{proof}

\begin{proposition}[Purely atomic criterion]\label{prop:atomic}
Assume that the restriction of \(\mug\) to \([t_0,t_1)\) is purely atomic, with atoms
\(\{\tau_j\}_{j\in J}\). Then
\[
 W(t_0,t_1)
 =
 \sum_{j\in J}
 \Deltag(\tau_j)H(\tau_j)H(\tau_j)^\top,
 \qquad
 H(\tau_j)=\Phi(t_1,\tau_j^+)B(\tau_j).
\]
The system is exactly controllable if and only if
\[
 \Span_{j\in J}\Range H(\tau_j)=\R^n.
\]
If \(J=\{1,\ldots,N\}\), this is equivalent to
\[
 \rank
 \begin{bmatrix}
 H(\tau_1)&H(\tau_2)&\cdots&H(\tau_N)
 \end{bmatrix}
 =n.
\]
\end{proposition}

\begin{proof}
Since \(\mugc=0\) on the control interval, the decomposition
\eqref{eq:W-decomp} gives
\[
 W=\sum_{j\in J}\Deltag(\tau_j)H(\tau_j)H(\tau_j)^\top.
\]
By Proposition~\ref{prop:propagator-measurable},
\[
 M_\Phi:=\sup_{s\in[t_0,t_1)}\norm{\Phi(t_1,s^+)}<\infty.
\]
Hence
\[
 \sum_{j\in J}\Deltag(\tau_j)\norm{H(\tau_j)}^2
 \le M_\Phi^2\sum_{j\in J}\Deltag(\tau_j)\norm{B(\tau_j)}^2
 =M_\Phi^2\norm{B}_{L_g^2}^2<\infty.
\]
Since \(\norm{H H^\top}\le\norm{H}^2\), this estimate proves absolute convergence of
the Gramian series in matrix norm.
For \(z\in\R^n\),
\begin{equation}\label{eq:atomic-gramian-qf}
 z^\top Wz
 =\sum_{j\in J}\Deltag(\tau_j)
  \norm{H(\tau_j)^\top z}^2.
\end{equation}
All summands are nonnegative and every mass is positive. Hence
\(z^\top Wz=0\) if and only if
\(H(\tau_j)^\top z=0\) for every \(j\in J\). This is equivalent to \(z\) being
orthogonal to
\(\Span_{j\in J}\Range H(\tau_j)\). The Gramian is positive definite exactly when
that span is \(\R^n\). For finite \(J\), the span condition is equivalent to the
rank of the displayed block matrix.
\end{proof}

\subsection{Finite atomic derivators and temporal richness}

For constant matrices \(A\in\R^{n\times n}\) and
\(B\in\R^{n\times m}\), define the truncated Krylov spaces
\begin{equation}\label{eq:krylov-spaces}
 \mathcal K_r(A,B)
 :=\operatorname{span}\{\operatorname{Ran}B,
 \operatorname{Ran}AB,\ldots,\operatorname{Ran}A^{r-1}B\},
 \qquad r\ge1,
\end{equation}
and \(\mathcal K_0(A,B):=\{0\}\). The controllability index used here is
\begin{equation}\label{eq:controllability-index}
 \nu(A,B):=
 \min\{r\ge1:\mathcal K_r(A,B)=\mathcal K_n(A,B)\}.
\end{equation}
If the pair is Kalman controllable, this is the least \(r\) for which
\(\mathcal K_r(A,B)=\R^n\). For a multi-input pair this is the largest of the
classical controllability indices; see, for example,
\citep{Kailath1980,Sontag1998}.

We begin with the polynomial structure generated by a finite atomic measure.

\begin{lemma}[A Newton-type polynomial basis]\label{lem:atomic-polynomial-basis}
Let \(\delta_1,\ldots,\delta_N\) be nonzero real numbers and define
\[
 p_j(z):=\prod_{k=j+1}^{N}(1+\delta_k z),
 \qquad j=1,\ldots,N,
\]
with the empty product \(p_N=1\). Then
\(p_1,\ldots,p_N\) form a basis of the polynomials of degree at most \(N-1\).
\end{lemma}

\begin{proof}
For \(j=1,\ldots,N\), the polynomial \(p_j\) has degree \(N-j\), and its leading
coefficient is
\[
 \gamma_j:=\prod_{k=j+1}^N\delta_k,
\]
with the convention \(\gamma_N=1\). Order the family as
\(p_N,p_{N-1},\ldots,p_1\), and express it in the monomial basis
\(1,z,\ldots,z^{N-1}\). The resulting coefficient matrix is upper triangular, and
its diagonal is
\[
 1,\ \delta_N,\ \delta_{N-1}\delta_N,\ \ldots,\
 \delta_2\delta_3\cdots\delta_N.
\]
Every diagonal entry is nonzero. Hence the coefficient matrix is invertible, so the
\(p_j\) are linearly independent. Since there are \(N\) of them in the
\(N\)-dimensional space of polynomials of degree at most \(N-1\), they form a basis.
\end{proof}

\begin{theorem}[Reachable subspace for a finite atomic derivator]
\label{thm:atomic-classification}
Assume that \(A\) and \(B\) are constant and that, on \([t_0,t_1)\),
\begin{equation}\label{eq:finite-atomic-clock}
 \mu_g=\sum_{j=1}^{N}\delta_j\delta_{\tau_j},
 \qquad
 t_0\le\tau_1<\cdots<\tau_N<t_1,
 \qquad \delta_j>0.
\end{equation}
Then
\begin{equation}\label{eq:atomic-krylov-identity}
 \mathcal R_g(t_0,t_1)=\mathcal K_N(A,B).
\end{equation}
Here \(\mathcal R_g(t_0,t_1)\) denotes the range of the forward endpoint map from
Proposition~\ref{prop:constant-forward-voc}, so the notation does not presuppose
regressivity. Consequently, the following are equivalent:
\begin{enumerate}[label=\textnormal{(\roman*)}]
\item the Stieltjes system is exactly controllable on \([t_0,t_1]\);
\item \(\operatorname{rank}[B,AB,\ldots,A^{N-1}B]=n\);
\item \((A,B)\) is Kalman controllable and \(N\ge\nu(A,B)\).
\end{enumerate}
In particular, the reachable subspace depends on the number of atoms, but not on their locations or positive masses. No regressivity assumption is required, since the proof uses only the forward atomic recurrence. In this setting, the post-jump propagator is the ordered product of the subsequent jump matrices.
\end{theorem}

\begin{proof}
Because the measure has no mass between consecutive atoms, the state is constant on
each interval \((\tau_j,\tau_{j+1}]\). Put \(u_j=u(\tau_j)\). At \(\tau_j\), the
canonical jump relation is
\begin{equation}\label{eq:atomic-recurrence}
 x(\tau_j^+)=(I+\delta_jA)x(\tau_j)+\delta_jBu_j.
\end{equation}
Starting from \(x(t_0)=x_0\) and iterating \eqref{eq:atomic-recurrence}, one obtains
\begin{align}
 x(t_1)
 ={}&\left[\prod_{k=1}^N(I+\delta_kA)\right]x_0 \notag\\
 &+\sum_{j=1}^N\delta_j
   \left[\prod_{k=j+1}^N(I+\delta_kA)\right]Bu_j.
 \label{eq:atomic-endpoint-recurrence}
\end{align}
The factors commute, but the displayed order also records their chronological action.
No inverse is used, so regressivity is unnecessary.

For reachability from the origin, the first term in
\eqref{eq:atomic-endpoint-recurrence} disappears. Since \(\delta_j>0\), scalar
multiplication by \(\delta_j\) does not change a column space. Thus
\[
 \cR_g(t_0,t_1)
 =\Span_{1\le j\le N}\Range p_j(A)B,
 \qquad
 p_j(z)=\prod_{k=j+1}^N(1+\delta_kz).
\]
Lemma~\ref{lem:atomic-polynomial-basis} implies both expansions
\[
 z^r=\sum_{j=1}^N\alpha_{rj}p_j(z),
 \qquad 0\le r<N,
\]
and
\[
 p_j(z)=\sum_{r=0}^{N-1}\beta_{jr}z^r.
\]
Evaluating at \(A\) and multiplying by \(B\) gives the two inclusions
\[
 \Span_{0\le r<N}\Range A^rB
 \subset
 \Span_j\Range p_j(A)B
 \subset
 \Span_{0\le r<N}\Range A^rB.
\]
Hence \(\cR_g=\mathcal K_N(A,B)\).

The system is exactly controllable if and only if this space equals \(\R^n\), which
is statement (ii). By definition of \(\nu(A,B)\), this is equivalent to Kalman
controllability together with \(N\ge\nu(A,B)\), proving (iii).
\end{proof}

\begin{corollary}[Finite time scales]\label{cor:finite-time-scale}
Let
\[
 \mathbb T=\{t_0<t_1<\cdots<t_N\}
\]
be a finite time scale and consider
\[
 x^\Delta(t)=Ax(t)+Bu(t).
\]
Then the reachable subspace at \(t_N\) from the origin is
\[
 \mathcal R_{\mathbb T}(t_0,t_N)
 =\Span\{\Range B,\Range AB,\ldots,\Range A^{N-1}B\}.
\]
Hence exact controllability in \(N\) steps is equivalent to
\[
 \rank[B,AB,\ldots,A^{N-1}B]=n.
\]
If \((A,B)\) is Kalman controllable, the least number of right-scattered control
instants needed for exact controllability equals \(\nu(A,B)\).
\end{corollary}

\begin{proof}
The right-scattered points in \([t_0,t_N)_{\mathbb T}\) are
\(t_0,\ldots,t_{N-1}\), with positive graininess
\(\delta_j=t_{j+1}-t_j\). Under the Stieltjes embedding these are precisely \(N\)
positive atoms, and the delta equation becomes
\[
 x(t_{j+1})=(I+\delta_jA)x(t_j)+\delta_jBu(t_j).
\]
The conclusion follows from Theorem~\ref{thm:atomic-classification}.
\end{proof}

\begin{remark}[Relation with finite-horizon discrete-time reachability]
The truncated Krylov space in Theorem~\ref{thm:atomic-classification} is the familiar
finite-horizon reachability space of a constant discrete-time system; see
\citep{HoKalmanNarendra1963,Kailath1980}. The additional observation here is that the
nonuniform step matrices \(I+\delta_jA\), for arbitrary nonzero masses, generate a
Newton basis and therefore exactly the same Krylov space. The masses affect the
Gramian and energy, but not the reachable subspace as long as they remain nonzero.
\end{remark}

\begin{corollary}[Minimum number of atoms for a purely atomic derivator]
\label{cor:min-atoms}
If \((A,B)\) is Kalman controllable, the minimum number of positive atoms among purely
atomic derivators that yields exact controllability is \(\nu(A,B)\).
\end{corollary}

\begin{remark}[Independence of the atomic masses]
Theorem~\ref{thm:atomic-classification} holds for every choice of positive masses. Once the number of atoms reaches the controllability index, all such choices yield the same reachable subspace. The Gramian and the minimum control energy, however, still depend on the masses.
\end{remark}

For \(\boldsymbol\delta=(\delta_1,\ldots,\delta_N)\in(0,\infty)^N\), write
\begin{equation}\label{eq:atomic-gramian-delta}
 W_N(\boldsymbol\delta)
 :=\sum_{j=1}^{N}\delta_j p_j(A)BB^\top p_j(A)^\top.
\end{equation}
This is a polynomial, hence continuous, matrix-valued function of the masses.
For clarity, the minimum-energy interpretation of this matrix does not require
regressivity. Indeed, identify the atomic control space with
\[
 \cU_N=(\R^m)^N,
 \qquad
 \norm{u}_{\cU_N}^2=\sum_{j=1}^N\delta_j\norm{u_j}^2,
\]
and define
\[
 \cK_Nu=\sum_{j=1}^N\delta_jp_j(A)Bu_j.
\]
A direct computation with the weighted inner product gives
\[
 (\cK_N^*z)_j=B^\top p_j(A)^\top z,
 \qquad
 \cK_N\cK_N^*=W_N(\boldsymbol\delta).
\]
Consequently, the Hilbert-space proof of Theorem~\ref{thm:gramian} and
Corollary~\ref{cor:minenergy} applies verbatim: \(W_N\) is positive definite exactly
when the atomic system is controllable, and the minimum energy required to generate a
displacement \(d\) is
\begin{equation}\label{eq:atomic-min-energy}
 \min_{\cK_Nu=d}\norm{u}_{\cU_N}^2
 =d^\top W_N(\boldsymbol\delta)^{-1}d.
\end{equation}

\begin{theorem}[Degeneration of the minimum control energy]
\label{thm:energy-degeneration}
Assume that \((A,B)\) is Kalman controllable and
\(\nu(A,B)=N\). Let
\(\boldsymbol\delta^{(r)}\in(0,\infty)^N\) converge to a vector
\(\bar{\boldsymbol\delta}\in[0,\infty)^N\) with exactly one zero component and all
other components positive. Then
\begin{equation}\label{eq:eigenvalue-degeneration}
 \lambda_{\min}\bigl(W_N(\boldsymbol\delta^{(r)})\bigr)\longrightarrow0,
 \qquad
 \bigl\|W_N(\boldsymbol\delta^{(r)})^{-1}\bigr\|\longrightarrow\infty.
\end{equation}
Consequently, the supremum of the minimum-energy cost over unit terminal displacements diverges.
\end{theorem}

\begin{proof}
The entries of \(W_N(\boldsymbol\delta)\) are polynomials in the masses, because each
\(p_j(A)\) is obtained by multiplying finitely many factors
\(I+\delta_kA\). Therefore
\[
 W_N(\boldsymbol\delta^{(r)})
 \longrightarrow W_N(\bar{\boldsymbol\delta})
\]
in matrix norm.

Suppose the \(\ell\)-th limiting mass is zero. The \(j=\ell\) term in
\eqref{eq:atomic-gramian-delta} vanishes. In every product corresponding to
\(j<\ell\), the factor \(I+\bar\delta_\ell A\) becomes \(I\); the products for
\(j>\ell\) are unchanged. After deleting the zero-mass atom and relabeling the
remaining ones, \(W_N(\bar{\boldsymbol\delta})\) is exactly the Gramian of a system
with \(N-1\) positive atoms. By Theorem~\ref{thm:atomic-classification}, its reachable
space is \(\mathcal K_{N-1}(A,B)\). The assumption \(\nu(A,B)=N\) means
\[
 \mathcal K_{N-1}(A,B)\subsetneq\mathcal K_N(A,B)=\R^n,
\]
so the limiting Gramian is singular.

Eigenvalues of real symmetric matrices depend continuously on the matrix entries;
see \citep{HornJohnson2013}. Hence
\[
 \lambda_{\min}(W_N(\boldsymbol\delta^{(r)}))
 \longrightarrow
 \lambda_{\min}(W_N(\bar{\boldsymbol\delta}))=0.
\]
For every \(r\), all masses are positive and \(N=\nu(A,B)\), so the prelimit Gramian
is positive definite. Consequently,
\[
 \norm{W_N(\boldsymbol\delta^{(r)})^{-1}}
 =\frac1{\lambda_{\min}(W_N(\boldsymbol\delta^{(r)}))}
 \longrightarrow\infty.
\]
By \eqref{eq:atomic-min-energy}, the minimum energy associated with a displacement
\(d\) is \(d^\top W_N^{-1}d\). The Rayleigh--Ritz formula \citep{HornJohnson2013} gives
\[
 \sup_{\norm{d}=1}d^\top W_N^{-1}d
 =\lambda_{\max}(W_N^{-1})
 =\norm{W_N^{-1}},
\]
which proves the last assertion.
\end{proof}

\begin{corollary}[Uniform lower bound for the atomic Gramian]
\label{cor:energy-compact}
If \((A,B)\) is Kalman controllable, \(N\ge\nu(A,B)\), and the masses range in a
compact box \([\varepsilon,M]^N\) with \(0<\varepsilon<M<\infty\), then there is a
constant \(c_{\varepsilon,M}>0\) such that
\[
 W_N(\boldsymbol\delta)\ge c_{\varepsilon,M}I
 \quad\text{for every }\boldsymbol\delta\in[\varepsilon,M]^N.
\]
Hence the minimum-energy cost for unit terminal displacements is uniformly bounded over this family of atomic measures.
\end{corollary}

\begin{proof}
For every \(\boldsymbol\delta\in[\varepsilon,M]^N\), all masses are positive and
\(N\ge\nu(A,B)\). Theorem~\ref{thm:atomic-classification} therefore implies that
\(W_N(\boldsymbol\delta)\) is positive definite. The function
\[
 \boldsymbol\delta\longmapsto
 \lambda_{\min}(W_N(\boldsymbol\delta))
\]
is continuous on the compact box \([\varepsilon,M]^N\). It consequently attains its
minimum there. Since it is strictly positive at every point, that minimum is a number
\(c_{\varepsilon,M}>0\). Thus
\[
 z^\top W_N(\boldsymbol\delta)z
 \ge c_{\varepsilon,M}\norm{z}^2
\]
for every \(z\in\R^n\) and every admissible mass vector. Equivalently,
\(W_N(\boldsymbol\delta)\ge c_{\varepsilon,M}I\). The energy bound follows from
\[
 d^\top W_N(\boldsymbol\delta)^{-1}d
 \le c_{\varepsilon,M}^{-1}\norm{d}^2.
\]
\end{proof}

\subsection{A mixed continuous--inactive--atomic double-integrator example}
\label{subsec:worked-example}

We illustrate the preceding results with a system containing continuous, inactive, and atomic phases. Consider the double integrator on \([0,3]\),
\begin{equation}\label{eq:worked-system}
 x'_g(t)=Ax(t)+Bu(t),\qquad
 A=\begin{pmatrix}0&1\\0&0\end{pmatrix},\qquad
 B=\begin{pmatrix}0\\1\end{pmatrix},
\end{equation}
with initial state \(x(0)=0\). Let
\begin{equation}\label{eq:worked-g}
 g(t)=
 \begin{cases}
 t,&0\le t\le1,\\
 1,&1<t\le2,\\
 t,&2<t\le3.
 \end{cases}
\end{equation}
Thus, the system evolves continuously on \([0,1]\), remains inactive on
\((1,2)\), undergoes a unit right jump at \(t=2\), and evolves continuously again
on \((2,3]\). More precisely,
\begin{equation}\label{eq:worked-measure}
 \mug=\mathcal L^1\!\restriction_{[0,1]}
      +\delta_2
      +\mathcal L^1\!\restriction_{(2,3]},
 \qquad \Deltag(2)=1.
\end{equation}
The value of a control on the inactive interval \((1,2)\) is irrelevant in
\(L_g^2\), whereas the value at \(t=2\) contributes to the endpoint map because
\(\mug(\{2\})=1\).

Since \(A^2=0\),
\[
 e^{\alpha A}=\Id+\alpha A,
 \qquad
 \Id+\Deltag(2)A=\Id+A=
 \begin{pmatrix}1&1\\0&1\end{pmatrix}.
\]
The matrix coefficient is therefore \(g\)-regressive. Formula
\eqref{eq:constant-formula} gives
\begin{equation}\label{eq:worked-Phi}
 \Phi(t,s)=\Id+\mug([s,t))A,
 \qquad 0\le s\le t\le3.
\end{equation}
For controllability at \(t_1=3\), the relevant quantity is the post-jump
propagator. Define
\[
 \vartheta(s):=\mug((s,3)).
\]
On the support of \(\mug\),
\begin{equation}\label{eq:worked-theta}
 \vartheta(s)=3-s,
 \qquad s\in[0,1]\cup\{2\}\cup(2,3).
\end{equation}
For \(s=2\), the interval \((2,3)\) excludes the atom at the lower endpoint, consistently with the post-jump notation \(s^+\). Hence
\begin{equation}\label{eq:worked-H}
 H(s)=\Phi(3,s^+)B
 =\bigl(\Id+\vartheta(s)A\bigr)B
 =\begin{pmatrix}\vartheta(s)\\1\end{pmatrix}
 =\begin{pmatrix}3-s\\1\end{pmatrix}
\end{equation}
for \(\mug\)-almost every \(s\).

The Gramian contains two continuous contributions and one atomic contribution:
\begin{align}
 W(0,3)
 &={}
 \int_0^1
 \begin{pmatrix}(3-s)^2&3-s\\3-s&1\end{pmatrix}\,ds
 +
 \begin{pmatrix}1&1\\1&1\end{pmatrix} \notag\\
 &\quad+
 \int_2^3
 \begin{pmatrix}(3-s)^2&3-s\\3-s&1\end{pmatrix}\,ds.      \label{eq:worked-W-integrals}
\end{align}
A direct calculation yields
\begin{equation}\label{eq:worked-W}
 W(0,3)=
 \begin{pmatrix}
 \dfrac{23}{3}&4\\[1mm]
 4&3
 \end{pmatrix},
 \qquad
 \det W(0,3)=7.
\end{equation}
Since \(\det W(0,3)>0\), the system is exactly controllable. The atomic term in \eqref{eq:worked-W-integrals} contributes the rank-one matrix associated with the instantaneous input at \(t=2\) and cannot be omitted.

We next consider the transfer from the origin to
\[
 x_1=\begin{pmatrix}1\\0\end{pmatrix}.
\]
Since \(\Phi(3,0)x(0)=0\), the displacement is \(d=x_1\), and
\begin{equation}\label{eq:worked-Winverse}
 W(0,3)^{-1}
 =\frac1{7}
 \begin{pmatrix}
 3&-4\\[1mm]
 -4&\dfrac{23}{3}
 \end{pmatrix},
 \qquad
 W(0,3)^{-1}d=
 \begin{pmatrix}3/7\\-4/7\end{pmatrix}.
\end{equation}
The minimum-energy control in Corollary~\ref{cor:minenergy} is therefore
\begin{equation}\label{eq:worked-control}
 u^*(s)=H(s)^\top W(0,3)^{-1}d
 =\frac{3(3-s)-4}{7}
 =\frac{5-3s}{7}
\end{equation}
for \(\mug\)-almost every active time. A convenient representative is
\begin{equation}\label{eq:worked-control-piecewise}
 u^*(s)=
 \begin{cases}
 \dfrac{5-3s}{7},&0\le s\le1,\\[1mm]
 0,&1<s<2,\\[1mm]
 -\dfrac17,&s=2,\\[1mm]
 \dfrac{5-3s}{7},&2<s<3.
 \end{cases}
\end{equation}
The value assigned on the inactive interval is merely a choice of representative, whereas the atomic value \(u^*(2)=-1/7\) is determined by the minimum-energy problem.

We next compute the corresponding trajectory. On the first continuous
phase, \eqref{eq:worked-system} reduces to the classical double integrator, and
\begin{equation}\label{eq:worked-trajectory-first}
 x(t)=
 \begin{pmatrix}
 \dfrac{t^2(5-t)}{14}\\[2mm]
 \dfrac{t(10-3t)}{14}
 \end{pmatrix},
 \qquad 0\le t\le1.
\end{equation}
Hence
\[
 x(1)=\begin{pmatrix}2/7\\1/2\end{pmatrix}.
\]
The state is constant throughout the inactive interval:
\begin{equation}\label{eq:worked-trajectory-plateau}
 x(t)=\begin{pmatrix}2/7\\1/2\end{pmatrix},
 \qquad 1<t\le2.
\end{equation}
At the atom, \(\Deltag(2)=1\), and the jump law reads
\[
 x(2^+)=(\Id+A)x(2)+Bu^*(2).
\]
Since
\[
 (\Id+A)x(2)
 =\begin{pmatrix}1&1\\0&1\end{pmatrix}
  \begin{pmatrix}2/7\\1/2\end{pmatrix}
 =\begin{pmatrix}11/14\\1/2\end{pmatrix},
 \qquad
 Bu^*(2)=\begin{pmatrix}0\\-1/7\end{pmatrix},
\]
we obtain
\begin{equation}\label{eq:worked-jump-state}
 x(2^+)=
 \begin{pmatrix}11/14\\5/14\end{pmatrix}.
\end{equation}
Finally, on \((2,3]\), integration from the post-jump state gives
\begin{equation}\label{eq:worked-trajectory-last}
 x(t)=
 \begin{pmatrix}
 \dfrac{-t^3+5t^2-3t+5}{14}\\[2mm]
 \dfrac{(3-t)(3t-1)}{14}
 \end{pmatrix},
 \qquad 2<t\le3.
\end{equation}
In particular,
\[
 x(3)=\begin{pmatrix}1\\0\end{pmatrix},
\]
as prescribed.

The minimum energy can be checked both from the Gramian and directly:
\begin{align}
 \norm{u^*}_{L_g^2}^2
 &=d^\top W(0,3)^{-1}d=\frac37,                    \label{eq:worked-energy-gramian}\\
 &=\int_0^1\left(\frac{5-3s}{7}\right)^2ds
   +\left(-\frac17\right)^2
   +\int_2^3\left(\frac{5-3s}{7}\right)^2ds       \notag\\
 &=\frac{13}{49}+\frac1{49}+\frac17
 =\frac37.                                         \label{eq:worked-energy-direct}
\end{align}
The example illustrates the role of the inactive interval, the post-jump propagator, the atomic contribution to the Gramian, and the distinction between control values on \(\mug\)-null sets and at atoms.

\section{Conclusions}

This paper separates two ingredients of exact controllability for linear Stieltjes
systems. The pair \((A,B)\) determines the algebraic control directions, whereas the
derivator determines which propagated directions are made available by the temporal
measure. The endpoint-operator formula and its Lebesgue--Stieltjes Gramian express this
separation for temporal structures containing continuous, singular-continuous,
inactive and atomic components.

For constant coefficients, and under regressivity, continuity and nonconstancy of the
derivator on one subinterval are sufficient to recover the classical Kalman criterion.
At the opposite extreme, a purely atomic derivator with exactly \(N\) positive atoms
produces the truncated Krylov space
\(\operatorname{span}\{\operatorname{Ran}B,\ldots,\operatorname{Ran}A^{N-1}B\}\).
Consequently, for a Kalman-controllable pair, the controllability index is the exact
minimum number of positive atoms among purely atomic derivators needed for
controllability. This algebraic statement is independent of the atomic locations and
masses.

The energy analysis shows that rank alone does not describe the quality of control.
Atomic masses enter the Gramian and determine its coercivity. For a fixed
Kalman-controllable pair with \(N=\nu(A,B)\), letting exactly one mass tend to zero
while the others converge to positive finite limits forces the smallest Gramian
eigenvalue to zero and makes the worst-case minimum energy unbounded. For a fixed
Kalman-controllable pair and \(N\ge\nu(A,B)\), masses in a compact subset of
\((0,\infty)^N\) yield a uniform positive Gramian lower bound. The worked
double-integrator example displays all of these mechanisms in a mixed
continuous--inactive--atomic setting.

Natural continuations include optimal allocation of a prescribed total Stieltjes
mass, optimal placement of atoms, time-dependent coefficients, robustness with respect
to perturbations of the derivator and the corresponding observability theory.

\section*{Funding}
This work was supported by the Xunta de Galicia through the project
``Consolidaci\'on e Estruturaci\'on 2023 GRC GI-1561---Ecuaci\'ons
diferenciais non lineais (EDNL).''

\section*{Acknowledgements}
During the preparation of this manuscript, the author used OpenAI's ChatGPT to assist
with language editing, LaTeX preparation, manuscript organization and consistency
checks. The author reviewed and verified the mathematical statements, proofs,
references and conclusions, and takes full responsibility for the content of the
manuscript.

\section*{Data availability}
No datasets were generated or analysed in this theoretical study.

\section*{Conflict of interest}
The author declares no conflict of interest.

\end{document}